\documentclass{IEEEtran}  

\IEEEoverridecommandlockouts                              

\def\TITLE{A distributed control strategy for reactive power compensation in smart microgrids}
\def\TITLEnl{A distributed control strategy for reactive\\power compensation in smart microgrids}
\def\AUTHORS{Saverio Bolognani and Sandro Zampieri}

\usepackage{amsmath}
\usepackage{amsfonts}
\usepackage{amssymb}
\usepackage{amsthm}

\newtheorem{theorem}{Theorem}
\newtheorem{corollary}[theorem]{Corollary}
\newtheorem{lemma}[theorem]{Lemma}
\newtheorem{proposition}[theorem]{Proposition}
\newtheorem*{remark}{Remark}
\newtheorem{definition}[theorem]{Definition}
\newtheorem{assumption}[theorem]{Assumption}

\usepackage{enumerate}

\usepackage{array}

\usepackage[english]{babel}

\usepackage[final]{graphicx} 

\usepackage[hyperindex=true, %
  pdftitle={\TITLE},%
  pdfauthor={\AUTHORS},%
  colorlinks=true,%
  pagebackref=false,%
  plainpages=false,%
  pdfpagelabels,%
  linkcolor=black,%
  citecolor=black,%
  filecolor=black,%
  urlcolor=black%
]{hyperref} 

\newlength{\IEEECOLWIDTH}
\newcommand{\expect}[2][{}]{\mathbb E_{#1} \left[ #2 \right]}

\def\1{{\mathbf{1}}}

\renewcommand\Re{\operatorname{Re}}
\renewcommand\Im{\operatorname{Im}}

\DeclareMathOperator{\Imag}{span}
\DeclareMathOperator{\vect}{vec}
\DeclareMathOperator{\diag}{diag}
\DeclareMathOperator{\trace}{Tr}

\DeclareMathOperator*{\argmin}{arg\,min}

\def\realnumbers{\mathbb{R}}
\def\complexnumbers{\mathbb{C}}

\def \graph	{\mathcal{G}}		
\def \nodes	{\mathcal{V}}		
\def \edges	{\mathcal{E}}		

\def \compensators {{\mathcal{C}}}		

\def \lapl {{\boldsymbol L}}
\def \green {{\boldsymbol X}}
\def \realgreen {{X}}

\def \cplxZ {{\boldsymbol Z}}

\def \PCC {{0}}

\def \nonodes {{n}}
\def \noedges {{|\edges|}}
\def \nocompensators {{m}}
\def \noclusters {{\ell}}

\def \Fave {{F_\text{ave}}}
\def \Eave {{E_\text{ave}}}

\title{\LARGE \bf \TITLEnl}

\author{\AUTHORS
\thanks{The authors are with the Department of Information Engineering, University of Padova, Italy. %
				Mailing address: via Gradenigo 6/B, 35131 Padova, Italy. %
				Phone: +39 049 827 7757. Fax +39 049 827 7614. %
        Email: {\tt\small \{saverio.bolognani, zampi\}@dei.unipd.it}.}%
}

\begin{document}

\maketitle



\begin{abstract}
We consider the problem of optimal reactive power compensation for the minimization
of power distribution losses in a smart microgrid.
We first propose an approximate model for the power distribution network,
which allows us to cast the problem into the class of convex quadratic, linearly constrained,
optimization problems.
We then consider the specific problem of commanding the microgenerators connected to the microgrid, in order to achieve the optimal injection of reactive power.
For this task, we design a randomized, gossip-like optimization algorithm.
We show how a distributed approach is possible,
where microgenerators need to have only a partial knowledge of the problem parameters and of the state, 
and can perform only local measurements.
For the proposed algorithm, 
we provide conditions for convergence
together with an analytic characterization of the convergence speed. 
The analysis shows that, in radial networks, the best performance can be achieved when we command cooperation among units that are neighbors in the electric topology.
Numerical simulations are included to validate the proposed model and to confirm the analytic results about the
performance of the proposed algorithm.
\end{abstract}




\section{Introduction}

Most of the distributed optimization methods have been derived 
for the problem of dispatching part of a large scale optimization
algorithm to different processing units \cite{Bertsekas_1997_ParallelandDistributed}.
When the same methods are applied to \emph{networked control systems} (NCS) \cite{NCS:07}, however, different issues arise.
The way in which decision variables are assigned to different agents 
is not part of the designer degrees of freedom.
Moreover, each agent has a local and limited knowledge of the problem parameters and of the system state.
Finally, the information exchange between agents can occur not only via a given communication channel, but also via local actuation and subsequent measurement performed on an underlying physical system.
The extent of these issues depends on the particular application.
In this work we present a specific scenario,
belonging to the motivating framework of smart electrical power distribution networks
\cite{Santacana_2010_GettingSmart, Ipakchi_2009_Gridoffuture},
in which these features play a central role.

In the last decade,
power distribution networks have seen 
the introduction of distributed microgeneration of electric energy
(enabled by technological advances and motivated by economical and environmental reasons).
This fact, 
together with an increased demand and the need for higher quality of service, 
has been driving the integration of a large amount of information and communication technologies (ICT) into these networks.
Among the many different aspects of this transition, 
we focus on the control of the distributed energy resources (DERs)
inside a smart microgrid
\cite{Lopes_2006_Definingcontrolstrategies,Green_2007_Controlofinverterbased}.
A microgrid is a portion of the low-voltage power distribution network that is managed autonomously from the rest of the network, 
in order to achieve better quality of the service, 
improve efficiency, and pursue specific economic interests.
Together with the loads connected to the microgrid 
(both residential and industrial customers),
we also have microgeneration devices 
(solar panels, combined heat-and-power plants, micro wind turbines, etc.).
These devices are connected to the microgrid via electronic interfaces (inverters),
whose main task is to enable the injection of the produced power into the microgrid.
However, these devices can also perform different other tasks,
denoted as \emph{ancillary services}
\cite{Katiraei_2006_Powermanagementstrategies,Prodanovic_2007_HarmonicandReactive,Tedeschi_2008_SynergisticControland}:
reactive power compensation, harmonic compensation, voltage support.

In this work we consider the problem of \emph{optimal reactive power compensation}.
Loads belonging to the microgrid may require a sinusoidal current which is not in phase with voltage.
A convenient description for this, consists in saying that they demand reactive power together with active power,
associated with out-of-phase and in-phase components of the current, respectively.
Reactive power is not a ``real'' physical power, 
meaning that there is no energy conversion involved nor fuel costs to produce it.
Like active power flows, 
reactive power flows contribute to power losses on the lines, 
cause voltage drop, and may lead to grid instability.
It is therefore preferable to minimize reactive power flows 
by producing it as close as possible to the users that need it.
We explore the possibility of using the electronic interface 
of the microgeneration units to optimize the flows of reactive power in the microgrid.
Indeed, the inverters of these units are generally oversized, 
because most of the 
distributed energy sources are intermittent in time, 
and the electronic interface is designed according to the peak power production.
When they are not working at the rated power, 
these inverters can be commanded to inject a desired amount 
of reactive power at no cost 
\cite{Rabiee_2009_ReactivePowerPricing}.

This idea has been recently investigated in the literature on power systems \cite{Yoshida_2000_particleswarmoptimization,Zhao_2005_multiagent-basedparticleswarm,Lavaei_2011_Powerflowoptimization}. 
However, these works consider a centralized scenario in which the parameters of the entire power grid are known, the controller can access the entire system state, the microgenerators are in small number, and they receive reactive power commands from a central processing unit.
In \cite{Rogers2010}, a hierarchical and secure architecture has been proposed for supporting such dispatching.
The contribution of this work consists in casting the same problem in the framework of networked control systems and distributed optimization.
This approach allows the design of algorithms and solutions which can guarantee scalability, robustness to insertion and removal of the units, and compliance with the actual communication capabilities of the devices.
Up to now, the few attempts of applying these tools to the power distribution networks
have focused on grids comprising a large number of mechanical synchronous generators, 
instead of power inverters.
See for example the stability analysis for these systems in \cite{Dorfler_2010_Synchronizationandtransient}
and the decentralized control synthesis in \cite{Guo_2000_Nonlineardecentralizedcontrol}.
Seminal attempts to distribute optimal reactive power compensation algorithms in a large-scale power network, consisted in dividing the grid into separate regions, each one provided with a supervisor \cite{Deeb1991,Kim1997}. Each of the supervisors can access all the regional measuments and data (including load demands) and can solve the corresponding small-scale optimization problem, enforcing consistency of the solution with the neighbor regions.  
Instead, preliminary attempts of designing distributed reactive power compensation strategies for microgrids populated by power inverters have been performed only very recently. However, the available results in this sense are mainly supported by heuristic approaches \cite{Tenti_submitted_Distributionlossminimization}, 
they follow suboptimal criteria for sharing a cumulative reactive power command among microgenerators \cite{Robbins2011},
and in some cases do not implement any communication or synergistic behavior of the microgenerators \cite{Turitsyn_2011_Optionscontrolreactive}.
In \cite{Baran2007}, a multi-agent-system approach has been employed for commanding reactive power injection in order to support the microgrid voltage profile.

The contribution of this paper is twofold. On one side, we propose a rigorous analytic derivation of an approximate model of the power flows. 
The proposed model can be considered as a generalization of the \emph{DC power flow model} commonly adopted in the power system literature (see for example \cite[Chapter 3]{Gomez_2009_Electricenergysystems} and references therein).
Via this model, the optimal reactive power flow (ORPF) problem is casted into a quadratic optimization (Section \ref{sec:problem_formulation}).
The second contribution is to propose and to analyze a 
distributed strategy for commanding the reactive power injection of each
microgeneration unit, capable of 
optimizing reactive power flows across the microgrid 
(Section \ref{sec:randomizedAlgorithm}). 
In Section \ref{sec:convergence} we characterize convergence of the proposed algorithm to the global optimal solution, and we study its performance.
In Section \ref{sec:optimalHypergraph} we show how the performance of this algorithm can be optimized by a proper choice of the communication strategy.
In Section \ref{sec:simulations}, we finally validate both the proposed model 
and the proposed algorithm via simulations.

\section{Mathematical preliminaries and notation}

Let $\graph= (\nodes, \edges, \sigma,\tau)$ be a directed graph, 
where $\nodes$ is the set of nodes,
$\edges$ is the set of edges, 
and $\sigma, \tau: \edges \rightarrow \nodes$ are two functions such that
edge $e \in \edges$ goes from the source node $\sigma(e)$ to the terminal node $\tau(e)$.
Two edges $e$ and $e'$ are \emph{consecutive} if $\{\sigma(e), \tau(e)\} \cap \{\sigma(e'), \tau(e')\}$ is not empty.
A \emph{path} is a sequence of consecutive edges.

In the rest of the paper we will often introduce complex-valued functions defined on the nodes and on the edges.
These functions will also be intended as vectors in $\complexnumbers^\nonodes$ (where $\nonodes = |\nodes|$) and $\complexnumbers^\noedges$.
Given a vector $u$, we denote by $\bar{u}$ its (element-wise) complex conjugate, and by $u^T$ its transpose. 

Let moreover $A \in \{0, \pm 1\}^{\noedges \times \nonodes}$ 
be the incidence matrix of the graph $\graph$, defined via its elements
$$
[A]_{ev} = \left\{\begin{array}{cl}
-1 & \text{if }v=\sigma(e) \\
1 & \text{if }v=\tau(e) \\
0 & \text{otherwise.}
\end{array}\right.
$$

If $\mathcal W$ is a subset of nodes, we define by $\1_{\mathcal W}$ the column vector whose elements are 
$$
[\1_{\mathcal W}]_{v}
\begin{cases}
1 & \text{if } v \in \mathcal W \\
0 & \text{otherwise}.
\end{cases}
$$
Similarly, if $w$ is a node, we denote by $\1_w$ the column vector 
whose value is $1$ in position $w$, and $0$ elsewhere,
and we denote by $\1$ the column vector of all ones.
If the graph $\graph$ is connected 
(i.e. for every pair of nodes there is a path connecting them),
then $\1$ is the only vector in the null space $\ker A$.

An undirected graph $\graph$ is a graph in which for every edge $e \in \edges$, 
there exists an edge $e' \in \edges$ such that $\sigma(e') = \tau(e)$ and $\tau(e') = \sigma(e)$.
In case the graph has no multiple edges and no self loops, 
we can also describe the edges of an undirected graph as subsets $\{h',h''\} \subseteq \nodes$ of cardinality 2.
Similarly, we define a hypergraph $\mathcal H$ 
as a pair $(\nodes, \edges)$ in which each edge is a 
subset $\{h', h'', \ldots \}$ of $\nodes$, of arbitrary cardinality
\cite{voloshin2009introduction}.


\section{Problem formulation}
\label{sec:problem_formulation}


\subsection{A model of a microgrid}

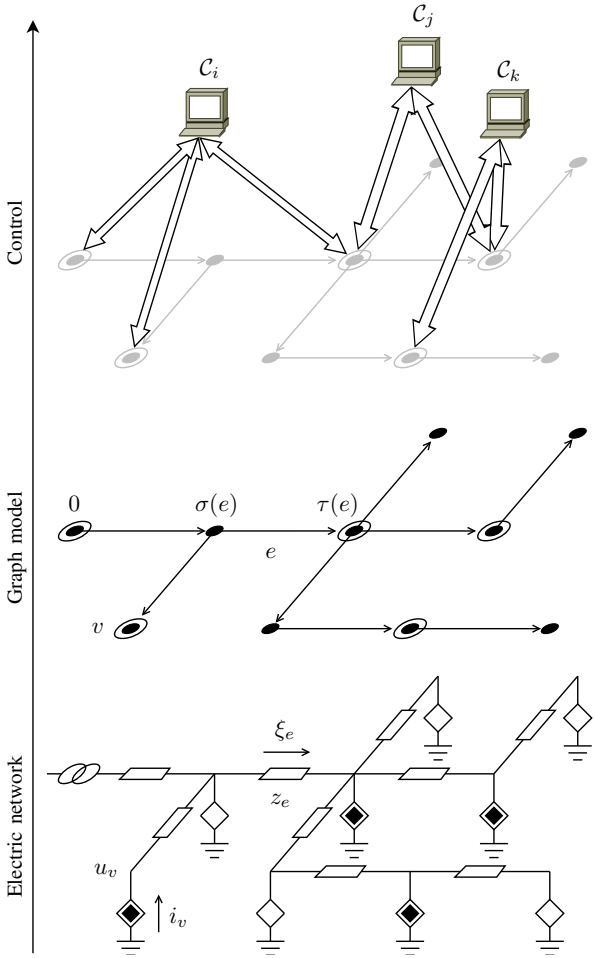
\begin{figure}[bt]
\centering
\resizebox{0.9\IEEECOLWIDTH}{!}{
\includegraphics[width=\IEEECOLWIDTH]{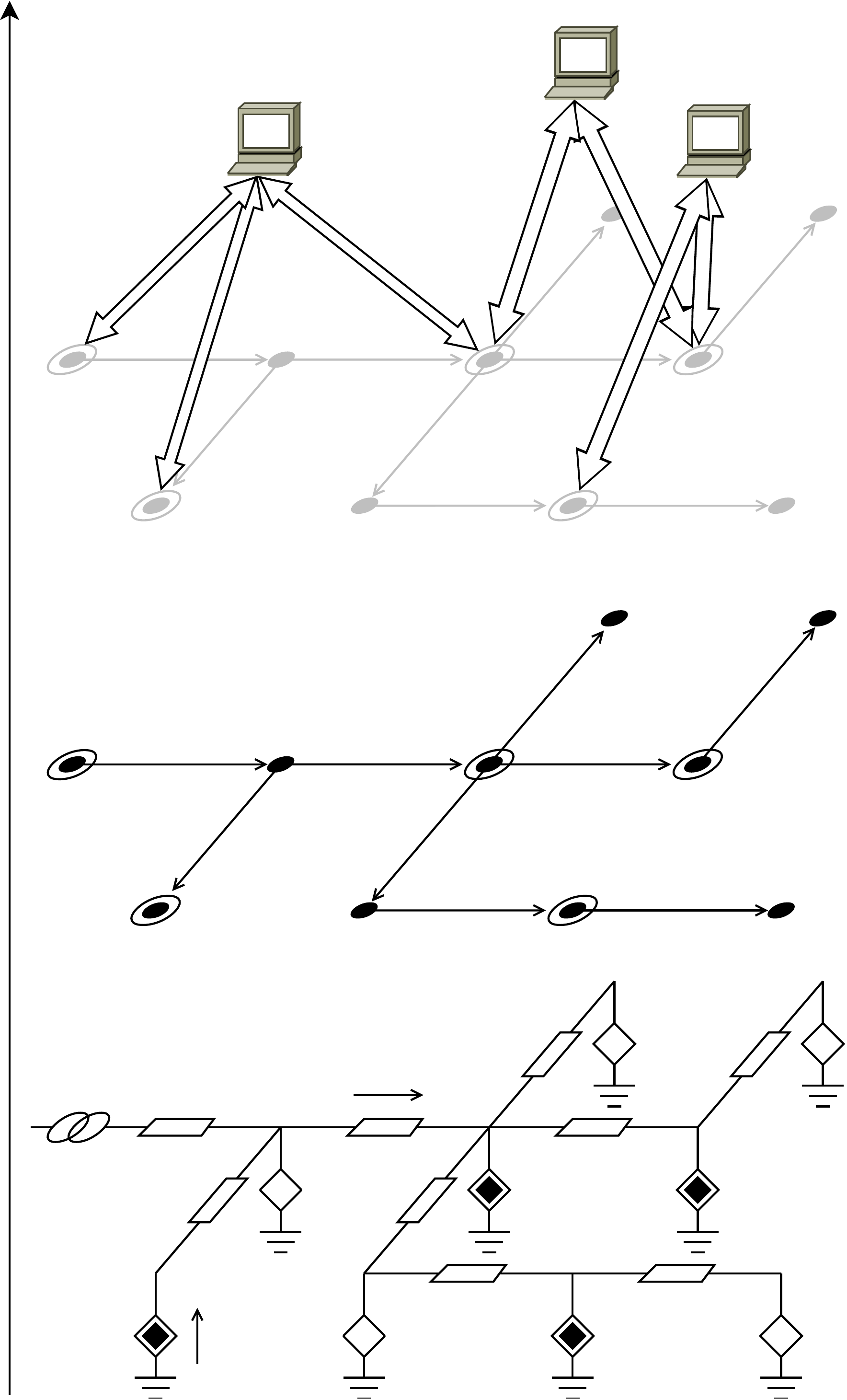}
\put(-250,25){\rotatebox{90}{\small Electric network}}
\put(-250,150){\rotatebox{90}{\small Graph model}}
\put(-250,295){\rotatebox{90}{\small Control}}
\put(-213,35){$u_v$}
\put(-181,14){$i_v$}
\put(-136,94){$\xi_e$}
\put(-138,65){$z_e$}
\put(-214,138){$v$}
\put(-170,190){$\sigma(e)$}
\put(-118,190){$\tau(e)$}
\put(-140,170){$e$}
\put(-224,190){$0$}
\put(-168,376){$\compensators_i$}
\put(-77,399){$\compensators_j$}
\put(-41,375){$\compensators_k$}
}
\caption{Schematic representation of the microgrid model.
In the lower panel a circuit representation is given, where black diamonds are microgenerators, white diamonds are loads, and the left-most element of the circuit represents the PCC.
The middle panel illustrates the adopted graph representation for the same microgrid. Circled nodes represent compensators (i.e. microgenerators and the PCC).
The upper panel shows how the compensators can be divided into overlapping clusters in order to implement the control algorithm proposed in Section \ref{sec:randomizedAlgorithm}. Each cluster is provided with a supervisor with some computational capability.
}
\label{fig:microgrid_model}
\end{figure}

We define a smart microgrid as a portion of the power distribution network that is connected to the power transmission network in one point, and hosts a number of loads and micro power generators, as described for example in \cite{Lopes_2006_Definingcontrolstrategies,Green_2007_Controlofinverterbased}
(see Figure~\ref{fig:microgrid_model}, lower panel).
For the purpose of this paper, we model a microgrid as a directed graph $\graph$, 
in which edges represent the power lines,
and nodes represent both loads and generators that are connected to the microgrid (see Figure~\ref{fig:microgrid_model}, middle panel).
These include loads, microgenerators,
and also the point of connection of the microgrid to the transmission grid
(called point of common coupling, or PCC).

We limit our study to the steady state behavior of the system,
when all voltages and currents are sinusoidal signals at the same frequency.
Each signal can therefore be represented via a complex number $y = |y|e^{j\angle y}$ whose absolute value $|y|$ corresponds to the signal
root-mean-square value, and whose phase $\angle y$ corresponds to the phase of the signal with respect to an arbitrary global reference.

In this notation, the steady state of a microgrid is described by the following system variables (see Figure~\ref{fig:microgrid_model}, lower panel):
\begin{itemize}
\item $u \in \complexnumbers^\nonodes$, where $u_v$ is the grid voltage at node $v$;
\item $i \in \complexnumbers^\nonodes$, where $i_v$ is the current injected by node $v$;
\item $\xi \in \complexnumbers^\noedges$, where $\xi_e$ is the current flowing on the edge $e$.
\end{itemize}

The following constraints are satisfied by $u,i$ and $\xi$
\begin{align}
A^T \xi + i & = 0, \label{eq:KCL} \\
Au + \cplxZ \xi &= 0, \label{eq:KVL}
\end{align}
where $A$ is the incidence matrix of $\graph$, and $\cplxZ = \diag(z_e, e\in \edges)$ is the diagonal matrix of line impedances, $z_e$ being the impedance of the microgrid power line corresponding to the edge $e$. 
Equation \eqref{eq:KCL} corresponds to Kirchhoff's current law (KCL) at the nodes, 
while \eqref{eq:KVL} describes the voltage drop on the edges of the graph.

Each node $v$ of the microgrid is then characterized by a law 
relating its injected current $i_v$ with its voltage $u_v$.
We model the PCC (which we assume to be node $\PCC$)
as an ideal sinusoidal voltage generator at the microgrid nominal voltage $U_N$ 
with arbitrary, but fixed, angle $\phi$
\begin{equation}
u_\PCC = U_N e^{j \phi}.
\label{eq:PCCmodel}
\end{equation}

We model loads and microgenerators (that is, every node $v$ of the microgrid except the PCC) via the following law relating the voltage $u_v$ and the current $i_v$
\begin{equation}
u_v\bar i_v = s_v\left|\frac{u_v}{U_N}\right|^{\eta_v},\quad \forall v \in \nodes \backslash \{ \PCC \},
\label{eq:ZIPModel}
\end{equation}
where $s_v$ is the \emph{nominal complex power} and $\eta_v$ is a characteristic parameter of the node $v$.
The model \eqref{eq:ZIPModel} is called \emph{exponential model} \cite{IEEE_1993_Loadrepresentationdynamic}
and is widely adopted in the literature on power flow analysis \cite{Haque_1996_Loadflowsolution}.
Notice that $s_v$ is the complex power that the node would inject into the grid,
if the voltage at its point of connection were the nominal voltage $U_N$.
The quantities
$$
p_v := \Re(s_v) \quad \text{and} \quad q_v := \Im(s_v) 
$$
are denoted as \emph{active} and \emph{reactive} power, respectively.
The nominal complex powers $s_v$ corresponding to microgrid loads are such that $\{p_v <0\}$, meaning that positive active power is \emph{supplied} to the devices. The nominal complex powers corresponding to microgenerators, on the other hand, are such that $\{p_v \ge 0\}$, as positive active power is \emph{injected} into the grid.
The parameter $\eta_v$ depends on the particular device. For example, constant power, constant current, and constant impedance devices are described by $\eta_v = 0,1,2$, respectively.
In this sense, this model is also a generalization of ZIP models \cite{IEEE_1993_Loadrepresentationdynamic},
which also are very common in the power system literature.
Microgenerators fit in this model with $\eta_v = 0$, as they generally are commanded via a complex power reference and they can inject it independently from the voltage at their point of connection
\cite{Lopes_2006_Definingcontrolstrategies,Green_2007_Controlofinverterbased}.

The task of solving the system of nonlinear equations given by 
\eqref{eq:KCL}, \eqref{eq:KVL}, \eqref{eq:PCCmodel}, and \eqref{eq:ZIPModel} to obtain the grid voltages
and currents, given the network parameters and
the injected nominal powers $\{s_v, v \in \nodes \backslash \{\PCC\}\}$ at every node, 
has been extensively covered in the literature under the denomination of \emph{power flow analysis}
(see for example \cite[Chapter 3]{Gomez_2009_Electricenergysystems}).
In the following, we derive an approximate model for the microgrid state, which will be used later for the
setup of the optimization problem and for the derivation of the proposed distributed algorithm.
To do so, a couple of technical lemmas are needed.

\begin{lemma}
Let $\lapl$ be the complex valued Laplacian $\lapl := A^T \cplxZ^{-1} A$. 
There exists a unique symmetric matrix $\green \in \complexnumbers^{\nonodes \times \nonodes }$ such that
\begin{equation}
\begin{cases}
\green \lapl = I - \1 \1_0^T \\
\green \1_0 = 0.
\end{cases}
\label{eq:Green_matrix_properties}
\end{equation}
\label{lemma:X}
\end{lemma}
\begin{proof}
Let us first prove the existence of $\green$.
As $\ker \lapl = \Imag \1 = \ker (I - \1 \1_0^T)$,
there exists $\green'\in\complexnumbers^{\nonodes \times \nonodes }$ such that $\green' \lapl =(I - \1 \1_0^T)$.
Let $\green = \green' (I - \1_0 \1^T)$. Then
\begin{align*}
\green \lapl &= \green' (I - \1_0 \1^T)\lapl = \green' \lapl = I - \1 \1_0^T, \\
\green \1_0  &= \green' (I - \1_0 \1^T) \1_0 = 0.
\end{align*}
Existence is then guaranteed.
To prove uniqueness, notice that
$$
\begin{bmatrix}\green & \1 \\ \1^T & 0 \end{bmatrix}
\begin{bmatrix}\lapl & \1_0 \\ \1_0^T & 0 \end{bmatrix}
=
\begin{bmatrix}\green \lapl + \1 \1_0^T & \green \1_0 \\ \1^T\lapl & \1^T \1_0 \end{bmatrix}
=
\begin{bmatrix}I & 0 \\ 0 & 1 \end{bmatrix}.
$$
Therefore
\begin{equation*}
\begin{bmatrix}\green & \1 \\ \1^T & 0 \end{bmatrix}
=
\begin{bmatrix}\lapl & \1_0 \\ \1_0^T & 0 \end{bmatrix}^{-1},
\label{eq:matricial_definition_of_green_matrix}
\end{equation*}
and uniqueness of $\green$ follows from the uniqueness of the inverse.
Moreover, as $\lapl = \lapl^T$, we have
$$
\begin{bmatrix}\green^T & \1 \\ \1^T & 0 \end{bmatrix}
=
\begin{bmatrix}\lapl & \1_0 \\ \1_0^T & 0 \end{bmatrix}^{-T}
=
\begin{bmatrix}\lapl^T & \1_0 \\ \1_0^T & 0 \end{bmatrix}^{-1}
= 
\begin{bmatrix}\green & \1 \\ \1^T & 0 \end{bmatrix}
$$
and therefore $\green = \green^T$.
\end{proof}

The matrix $\green$ depends only on the topology of the microgrid power lines and on their impedance
(compare it with the definition of Green matrix in \cite{Ghosh_2008_Minimizingeffectiveresistance}).
Indeed, 
it can be shown that, for every pair of nodes $(v,w)$,
\begin{equation}
(\1_v - \1_w)^T \green (\1_v - \1_w) = Z^\text{eff}_{vw},
\label{eq:effectiveResistance}
\end{equation}
where $Z^\text{eff}_{vw}$ represents the \emph{effective impedance} of the power lines between node $v$ and $w$.

All the currents $i$ and the voltages $u$ of the microgrid are therefore determined by the equations
\begin{equation}
\begin{cases}
u=\green i+ U_N e^{j\phi} \1 \\ 
\1^T i = 0\\
\displaystyle u_v\bar i_v = s_v\left|\frac{u_v}{U_N}\right|^{\eta_v},\quad \forall v \in \nodes \backslash \{ \PCC \},
\end{cases}
\label{eq:det_iu_U0}
\end{equation}
where the first equation results from \eqref{eq:KCL}, \eqref{eq:KVL}, and \eqref{eq:PCCmodel}
together with Lemma \ref{lemma:X},
while the second equation descends from \eqref{eq:KCL}, using the fact that $A\1 =0$ in a connected graph.

We can see the currents $i$ and the voltages $u$ as functions $i(U_N),u(U_N)$ of $U_N$.
The following proposition provides the Taylor approximation of $i(U_N)$ and $u(U_N)$ for large $U_N$.

\begin{proposition}
Let $s$ be the vector of all nominal complex powers $s_v$, 
including 
\begin{equation}
s_0:=-\sum_{v\in\nodes\backslash\{\PCC\}} s_v.
\label{eq:def_s0}
\end{equation}
Then for all $v \in \nodes$ we have that
\begin{equation}
\begin{split}
i_v(U_N) &= e^{j\phi} \left( \frac{\bar s_v}{U_N} +  \frac{c_v(U_N)}{U_N^2} \right) \\
u_v(U_N) &= e^{j\phi} \left( U_N + \frac{[\green \bar s]_v}{U_N} 
+ \frac{d_v(U_N)}{U_N^2}\right) 
\end{split}
\label{eq:approximate_solution}
\end{equation}
for some complex valued functions $c_v(U_N)$ and $d_v(U_N)$ which are $O(1)$ as $U_N\to\infty$, i.e. they are bounded functions for large values of the 
nominal voltage $U_N$.
\label{proposition:taylor_expansion_complex}
\end{proposition}

The proof of this proposition is based on elementary multivariable analysis, but requires a quite involved notation. For this reason it is given in Appendix~\ref{app:taylor}.
The quality of the  approximation proposed in the previous proposition relies on having large nominal voltage $U_N$
and relatively small currents injected by the inverters (or supplied to the loads).
This assumption is verified in practice and corresponds to correct design and operation of power distribution networks,
where indeed the nominal voltage is chosen sufficiently large 
(subject to other functional constraints) 
in order to deliver electric power to the loads with relatively small power losses on the power lines.
Numerical simulation in Section \ref{sec:simulations} will indeed show that 
the approximation is extremely good when power distribution networks are operated in their regular regime.

\begin{remark}
Notice that this approximation, in similar fashions, has been used in the literature before 
for the problem of estimating power flows on the power lines
(see among the others \cite{Baran_1989_Optimalsizingcapacitors,Turitsyn_2011_Optionscontrolreactive} 
and references therein).
It also shares some similarities with the DC power flow model \cite[Chapter 3]{Gomez_2009_Electricenergysystems}, extending it to the lossy case (in which lines are not purely inductive).
The analytical rigorous justification proposed here allows us to estimate the approximation error, 
and, more importantly, will provide the tools to understand what information on the system state 
can be gathered by properly sensing the microgrid.
\end{remark}

\subsection{Power losses minimization problem}

Similarly to what has been done for example in 
\cite{Cagnano_toappear_On-lineoptimalreactive},
we choose active power losses on the power lines as a metric for optimality of reactive power flows.
The total active power losses on the edges are given by
\begin{equation}
J^\text{tot} := \sum_{e\in \edges} |\xi_e|^2 \Re(z_e)
= \bar i^T \Re(\green) i,
\label{eq:Jtot}
\end{equation}
where we used \eqref{eq:KVL} and \eqref{eq:det_iu_U0}, together with the properties \eqref{eq:Green_matrix_properties} of $\green$, to express $\xi$ as a function of $i$.

In the scenario that we are considering, 
we are allowed to command only a subset $\compensators  \subset \nodes$
of the nodes of the microgrid (namely the microgenerators, also called \emph{compensators} in this framework).
We denote by $\nocompensators:=|\compensators|$ its cardinality.
Moreover, we assume that for these compensators we are only allowed to 
command the amount of reactive power $q_v = \Im(s_v)$ injected into the grid, as the decision on the amount of active power
follows imperative economic criteria. For example, in the case of renewable energy sources,
any available active power is typically injected into the grid to replace generation from traditional plants,
which are more expensive and exhibit a worse environmental impact 
\cite{Rabiee_2009_ReactivePowerPricing}.

The resulting optimization problem is then
\begin{equation}
\min_{q_v, \, v\in \compensators}\;
J^\text{tot},
\label{eq:nl_optimization_problem}
\end{equation}
where the vector of currents $i$ is a function of the decision variables $q_v$, $v \in \compensators$, via the implicit system of nonlinear equations \eqref{eq:det_iu_U0}.

In the typical formulations of the optimal reactive power compensation problem, some constraints might be present.
Traditional operational requirements constrain the voltage amplitude at every node to stay inside a given range centered around the microgrid nominal voltage
$$
|u_v| \in \left[ U_N - \Delta U, \; U_N + \Delta U \right], \quad v \in \nodes.
$$
Moreover, the power inverters that equip each microgenerator,
due to thermal limits, can only provide a limited amount of reactive power.
This limit depends on the size of the inverter but also on the 
concurrent production of active power by the same device 
(see \cite{Turitsyn_2011_Optionscontrolreactive}).
These constraints can be quite tight, and correspond to a set of \emph{box constraint} on the decision variables,
$$
q_v \in \left[ q_v^\text{min}, \; q_v^\text{max} \right], \quad v \in \compensators \backslash \{\PCC\}.
$$
These constraints have been relaxed
for the analysis presented in this paper, 
and for the subsequent design of a control strategy.
This choice allowed to conduct an analytic study of the performance of the proposed solution.
One possible extension of the approach presented here, dealing effectively with such limits, has been proposed in \cite{Bolognani2012}.

In the following, we show how the approximated model proposed in the previous section can be used to tackle the optimization problem \eqref{eq:nl_optimization_problem} and to design an algorithm for its solution.
By plugging the approximate system state \eqref{eq:approximate_solution}
into \eqref{eq:Jtot},
we have
\begin{align*}
J^\text{tot} &= \frac{1}{U_N^2} \bar s^T \Re(\green) s + \frac{1}{U_N^3}\tilde J(U_N,s) \\
&= \frac{1}{U_N^2} p^T \Re(\green) p + \frac{1}{U_N^2} q^T \Re(\green) q + \frac{1}{U_N^3}\tilde J(U_N,s)
\end{align*}
where $p = \Re(s)$, $q = \Im(s)$, and 
\begin{equation*}
\begin{split}
\tilde J(U_N,s) &:= 2\Re\left[s^T \Re(\green) c(U_N)\right]\\
&\phantom{:=} +\frac{1}{U_N}\bar c^T(U_N) \Re(\green) c(U_N)
\end{split}
\end{equation*} 
is bounded as $U_N$ tends to infinity. 
The term $\frac{1}{U_N^3}\tilde J(U_N,s)$ 
can thus be neglected if $U_N$ is large.

Therefore, via the approximate model \eqref{eq:approximate_solution},
we have been able to 
\begin{itemize}
\item approximate power losses as a quadratic function of the injected power;
\item decouple the problem of optimal power flows into the problem of optimal active and reactive power injection.
\end{itemize}

The problem of optimal reactive power injection at the compensators can therefore be
expressed as a quadratic, linearly constrained problem, in the form
\begin{equation}
\begin{split}
\min_{q_v, v\in \compensators} & \quad J(q), \qquad \text{where}\quad J(q) = \frac12 q^T \Re(\green) q,\\
\text{subject to}
& \quad \1^T q = 0 
\end{split}
\label{eq:Qoptimizationproblem}
\end{equation}
where the other components of $q$, namely $\{q_v, v \in \nodes \backslash \compensators\}$, are the nominal amounts of reactive power injected by the nodes that cannot be commanded,
and the constraint $\1^T q = 0$ directly descends from \eqref{eq:def_s0}.

The solution of the optimization problem \eqref{eq:Qoptimizationproblem} 
would not pose any challenge
if a centralized solver knew the problem parameters $\green$ and $\{q_v, v \in \nodes \backslash \compensators \}$. 
These quantities depend both on the grid parameters and on the reactive power demand of all the microgrid loads.
In the case of a large scale system, collecting all this information at a central location 
would be unpractical for many reasons, including communication constraints, delays, reduced robustness, and privacy concerns.

\section{A distributed algorithm for reactive power dispatching}
\label{sec:randomizedAlgorithm}

In this section, we present an algorithm that
allows the compensators to decide on the amount of reactive power 
that each of them have to inject in the microgrid in order to minimize the power distribution losses.
The proposed approach is based on a distributed strategy,
meaning that it can be implemented by the compensators
without the supervision of any central controller.
In order to design such a strategy, the optimization problem is decomposed into smaller, tractable subproblems that are assigned to small groups of compensators
(possibly even pairs of them).
We will show that the compensators can solve their corresponding subproblems via local measurements, a local knowledge of the grid topology, and a limited data processing and communication. We will finally show that the repeated solution of these subproblems yields the solution of the original global optimization problem.

It is worth remarking that the decomposition methods proposed in most of the literature on distributed optimization, for example in \cite{Nedic_2009_Distributedsubgradientmethods}, cannot be applied to this problem because the cost function in \eqref{eq:Qoptimizationproblem} is not separable into a sum of individual terms for each agent.
Approaching the decomposition of this optimization problem via its dual formulation (as proposed in many works, including the recent \cite{Boyd_2011_DistributedoptimizationADMM}) is also unlikely to succeed, as the feasibility of the system state must be guaranteed at any time during the optimization process.

\subsection{Optimization problem decomposition}
\label{subsec:problemdecomposition}

\def\cli{r} 

Let the compensators be divided into $\ell$ possibly overlapping subsets
$\mathcal C_1,\ldots, \mathcal C_\noclusters$
with $\bigcup_{\cli=1}^\noclusters \mathcal C_\cli = \compensators$. 
Nodes belonging to the same subset (or \emph{cluster}) are able to communicate each other, and they are therefore capable of coordinating their actions and 
sharing their measurements.
Each cluster is also provided with some computational capabilities for processing the collected data. 
This processing is performed by a cluster supervisor, capable of collecting all the necessary data from the compensators belonging to the cluster and to return the result of the data processing to them (see Figure~\ref{fig:microgrid_model}, upper panel). The cluster supervisor can possibly be one of the compensators in the cluster. 
An alternative implementation consists in providing all the compensators in the cluster with identical instances of the same instructions. In this case, after sharing the necessary data, they will perform exactly the same data processing, and no supervising unit is needed.

The proposed optimization algorithm consists of the following repeated steps, which occur at time instants $T_t\in{\realnumbers}$, $t=0,1,2,\ldots$.
\begin{enumerate}
\item A cluster $\mathcal C_{\cli(t)}$ is chosen, where $\cli(t) \in \{1,\ldots, \noclusters\}$.
\item The compensators in $\mathcal C_{\cli(t)}$, by sharing their state (the injected reactive powers $q_k$, $k\in \mathcal C_{\cli(t)}$) and their measurements (the voltages $u_k$, $k\in \mathcal C_{\cli(t)}$), determine their new feasible state that minimizes the global cost $J(q)$, solving the optimization \emph{subproblem} in which the nodes not belonging to $\mathcal C_{\cli(t)}$ keep their state constant.
\item The compensators in $\mathcal C_{\cli(t)}$
update their states $q_k$, $k\in \mathcal C_{\cli(t)}$ by injecting the new reactive powers computed in the previous step.
\end{enumerate}

In the following, we provide the necessary tools for implementing these steps. In particular, we show how the compensators belonging to the cluster $\compensators_{\cli(t)}$ can update their state to minimize the total power distribution losses based only on their partial knowledge of the electrical network topology and on the measurements they can perform.

Consider the optimization \eqref{eq:Qoptimizationproblem}, and let us distinguish the controllable and the uncontrollable components of the vector $q$.
%
%
Assume with no loss of generality that the first $\nocompensators$ components of $q$ are controllable
(i.e. they describe the reactive power injected by the compensators)
and that the remaining $\nonodes-\nocompensators$ are not controllable.
The state $q$ is thus partitioned as
$$
q=\begin{bmatrix}q_{\mathcal C}\\ q_{\nodes \backslash {\mathcal C}}\end{bmatrix}
$$
where $q_{\mathcal C}\in\realnumbers^\nocompensators$ and $q_{\nodes \backslash {\mathcal C}}\in\realnumbers^{\nonodes-\nocompensators}$.
According to this partition of $q$, let us also partition the matrix $\Re(\green)$ as
\begin{equation}
\Re(\green)
=
\left[\begin{array}{cc}
	M & N \\ 
	N^T & Q
\end{array}\right].
\label{eq:blockpartitioning}
\end{equation}

Let us also introduce some convenient notation. Consider the subspaces
$$
\mathcal S_\cli:=
\left\{
q_{\mathcal C} \in \realnumbers^\nocompensators\ :\ 
\sum_{j\in\compensators_r} \left[q_{\mathcal C}\right]_j = 0\ ,
\left[q_{\mathcal C}\right]_j = 0\ \forall j\not\in\compensators_\cli
\right\},
$$
and the $\nocompensators \times \nocompensators$ matrices
\begin{align*}
\Omega	&:=\frac{1}{2m}\sum_{h,k \in \compensators}(\1_h-\1_k)(\1_h-\1_k)^T = I-\frac{1}{\nocompensators}\1\1^T, \\
\Omega_\cli &:=\frac{1}{2|\mathcal C_\cli|}\sum_{h,k\in\mathcal C_r}(\1_h-\1_k)(\1_h-\1_k)^T \\
	&= \diag(\1_{\compensators_\cli}) - \frac{1}{|\compensators_\cli|}\1_{\compensators_\cli}\1_{\compensators_\cli}^T,
\end{align*}
where $\diag(\1_{\compensators_\cli})$ is the $m \times m$ diagonal matrix whose diagonal is the vector $\1_{\compensators_\cli}$.

Notice that $\mathcal S_\cli = \Imag \Omega_\cli$.
We list some useful properties:
\begin{enumerate}[i)]
\item $\Omega_\cli^2=\Omega_\cli$ and $\Omega_\cli^\sharp = \Omega_\cli$, where $\sharp$ means pseudoinverse.
\item The matrices $\Omega_\cli$, $\Omega_\cli M \Omega_\cli$ and $(\Omega_\cli M \Omega_\cli)^\sharp$ have entries different from zero only in position $h,k$ with $h,k\in \mathcal C_\cli$.
\item  $\ker (\Omega_\cli M \Omega_\cli)^\sharp = \ker \Omega_\cli$
and $\Imag(\Omega_\cli M \Omega_\cli)^\sharp = \Imag \Omega_\cli$, and thus
$(\Omega_\cli M \Omega_\cli)^\sharp = (\Omega_\cli M \Omega_\cli)^\sharp \Omega_\cli$.
\end{enumerate}

\def\qopti{q^\text{opt, $\cli$}_{\mathcal C}}

If the system is in the state $q =\left[\begin{smallmatrix}q_{\mathcal C}\\ q_{\nodes \backslash {\mathcal C}}\end{smallmatrix}\right]$ and cluster $\mathcal C_\cli$ is activated,
the corresponding cluster supervisor has to solve the optimization problem 
\begin{equation}
\begin{split}
\qopti := \argmin_{q'_{\mathcal C}}		 
	&\quad	J\left(\begin{bmatrix}q'_{\mathcal C}\\ q_{\nodes \backslash {\mathcal C}}\end{bmatrix}\right)\\
	\text{subject to}& \quad q'_{\mathcal C} - q_{\mathcal C}\in \mathcal S_\cli.
\end{split}
\label{eq:optimization_subproblem}
\end{equation}
Using the standard formulas for quadratic optimization it can be shown that 
\begin{equation}\label{eq:argmin23}
\begin{split}
\qopti
	&=
	q_{\mathcal C} - (\Omega_\cli M \Omega_\cli)^\sharp \nabla J,
\end{split}
\end{equation}
where 
\begin{equation}
\nabla J 
= Mq_{\mathcal C}+ N q_{\nodes \backslash {\mathcal C}} 
\label{eq:gradientExpression}
\end{equation} 
is the gradient of $J\left(\left[\begin{smallmatrix}q_{\mathcal C}\\ q_{\nodes \backslash {\mathcal C}}\end{smallmatrix}\right]\right)$
with respect to the decision variables $q_{\mathcal C}$.

Notice that, 
by using property ii) of the matrices $(\Omega_\cli M \Omega_\cli)^\sharp$, 
we have
\begin{equation}
\left[ \qopti \right]_h = \!
\begin{cases}
q_h
-
\sum_{k\in\mathcal C_\cli}
\left[\left(\Omega_\cli M \Omega_\cli \right)^\sharp\right]_{hk}
\left[\nabla J\right]_k
&
\!\!\text{if $h \in \compensators_\cli$} \\
q_h
&
\!\!\text{if $h \notin \compensators_\cli$.}
\end{cases}
\label{eq:node_law}
\end{equation}

In the following, we show how the supervisor of $\mathcal C_\cli$ can perform an approximate computation of $\left[ \qopti \right]_h, \; h \in \compensators_\cli$,
based on the information available within $\mathcal C_\cli$.

\subsection{Hessian reconstruction from local topology information}

Compensators belonging to the cluster $\compensators_\cli$  can infer some information about the Hessian $M$.
More precisely they can determine the elements of the matrix $(\Omega_\cli M \Omega_\cli)^\sharp$ appearing in \eqref{eq:node_law}. 
Define $R^\text{eff}$ as a $m\times m$ matrix with entry in $h,k \in \compensators$ equal to $\Re\left(Z^\text{eff}_{hk}\right)$, 
where $Z^\text{eff}_{hk}$ are the mutual effective impedances  between pairs compensators $h,k$. 
From \eqref{eq:effectiveResistance} and \eqref{eq:blockpartitioning} we have that
\begin{equation}
[R^\text{eff}]_{hk}=\Re\left(Z^\text{eff}_{hk}\right)=(\1_h-\1_k)^T M(\1_h-\1_k)
\label{eq:ReffM}
\end{equation}
and so we can write 
\begin{equation*}
\begin{split}
R^\text{eff} &= \sum_{h,k\in\mathcal C}\1_h(\1_h-\1_k)^TM(\1_h-\1_k)\1_k^T \\
			&=\diag(M)\1\1^T+\1\1^T\diag(M)-2 M 
\end{split}
\end{equation*}
where $\diag(M)$ is the $m\times m$ diagonal matrix having the same diagonal elements of $M$.
Consequently, since $\Omega_\cli \1=0$, we have that $\Omega_\cli R^\text{eff} \Omega_\cli=-2\Omega_\cli M \Omega_\cli$
and then
\begin{equation}\label{eq:HdaRsimpl}
\Omega_\cli M \Omega_\cli=-\frac{1}{2}\Omega_\cli R^\text{eff} \Omega_\cli.
\end{equation}
Because of the specific sparseness of the matrices $\Omega_\cli$
and via \eqref{eq:ReffM}, we can then argue that the elements of the matrix $\Omega_\cli M \Omega_\cli$ can be computed by the supervisor of cluster $\compensators_\cli$ from the mutual effective impedances $Z^\text{eff}_{hk}$ for $h,k\in\compensators_\cli$. These impedances are assumed to be known by the cluster supervisor. This is a reasonable hypothesis, since the mutual effective impedances can be obtained via online estimation procedures as in \cite{Teodorescu_2007} or via ranging technologies over power line communications as suggested in \cite{Costabeber_2011_Ranging}. 
Moreover, in the common case in which the low-voltage power distribution grid is radial, the mutual effective impedances $Z^\text{eff}_{hk}$ correspond to the impedance of the only electric path connecting node $h$ to node $k$, and can therefore be inferred from \emph{a priori} knowledge of the local microgrid topology.

Then the cluster supervisor needs to compute the pseudo-inverse of $\Omega_\cli M \Omega_\cli$.
Notice that all these operations have to be executed offline once, as these coefficients depend only on the grid topology and impedances,
and thus are not subject to change.

\subsection{Gradient estimation via local voltage measurement}

Assume that nodes in $\compensators_\cli$ can measure the grid voltage $u_k$, $k\in \mathcal C_\cli$, at their point of connection. In practice, this can be done via \emph{phasor measurement units} that return both the amplitudes $|u_k|$ and phases $\angle u_k$ of the measured voltages (see 
\cite{Phadke_1993_Synchronizedphasormeasurements,Phadke_2008_Synchronizedphasormeasurements}).
Let us moreover assume the following.

\begin{assumption}
\label{ass:theta}
All power lines in the microgrid have the same inductance/resistance ratio, i.e.
$$
\cplxZ = e^{j\theta} Z
$$
where $Z$ is a diagonal real-valued matrix,
whose elements are $[Z]_{ee} = |z_e|$.
Consequently, $\lapl = e^{-j \theta} A^T Z^{-1}A$, and $\realgreen := e^{-j\theta}\green$ is a real-valued matrix.
\end{assumption}

This assumption seems reasonable in most practical cases 
(see for example the IEEE standard testbeds \cite{Kersting2001}).
The effects of this  approximation will be commented in Section \ref{sec:simulations}.


\def\uc{u_\compensators}

Similarly to the definition of $q_\compensators$, we define by $\uc$ the vector in $\complexnumbers^\nocompensators$ containing the components $u_k$ with 
$k\in \compensators$, i.e. the voltages that can be measured by the compensators.
Consider now the maps $K_r:\complexnumbers^\nocompensators \to \complexnumbers^\nocompensators$, $r=1,\ldots,\ell$, defined component-wise as follows
\begin{multline}
\left[K_r(\uc)\right]_k:= \\
\begin{cases}
\frac{1}{|\compensators_\cli|}\sum_{v \in \compensators_\cli} 
|u_v||u_k|\sin(\angle u_v-\angle u_k-\theta) 
& \text{if $k\in \compensators_\cli$}\\
0
& \text{if $k\not\in \compensators_\cli$.}
\end{cases}
\label{eq:def_K}
\end{multline}
Notice first that $K_r(\uc)$ can be computed locally by the supervisor of the cluster $\cli$, from the voltage measurements of the compensators belonging to $\compensators_\cli$. The following result shows how the function $K_\cli(\uc)$ is related with the gradient of the cost function, introduced in \eqref{eq:gradientExpression}.

\begin{proposition}
Let $K_r(\uc)$ be defined as in \eqref{eq:def_K}.
Then the elements $[\nabla J]_k$, $k \in \compensators_\cli$, 
of the gradient introduced in (\ref{eq:gradientExpression}) 
can be decomposed as
\begin{equation}
 \left[ \nabla J \right]_k 
 =
 \cos\theta \left[K_r(\uc)\right]_k 
+
\alpha_r 
+
\frac{1}{U_N}
\left[\tilde K_r(\uc)\right]_k,
\label{eq:Kdecomp}
\end{equation}
where
$\left[\tilde K_r(\uc)\right]_k$ is bounded as $U_N \rightarrow \infty$,
and $\alpha_r$ is a constant independent of the compensator index $k$.
\label{pro:Kdecomp}
\end{proposition}
\begin{proof}
It can be seen that, for $k \in \compensators_\cli$,
$$
\left[K_r(\uc)\right]_k 
= -\Im \left[  e^{-j\theta} \frac{1}{|\compensators_\cli|} 
\left(\1_{\compensators_\cli}^T \bar u \right) u_k \right].
$$
By using \eqref{eq:approximate_solution}, and via a series of simple but rather tedious computation it can be seen that \eqref{eq:Kdecomp} holds with
\begin{align*}
\alpha_r &= 
\frac{\cos\theta }{|\compensators_\cli|} \Im \left[  e^{-j2\theta} 
\left(\1_{\compensators_\cli}^T Xs \right) \right]\\
&\quad +\frac{\cos\theta }{|\compensators_\cli|U_N} \Im \left[e^{-j\theta}
\left(\1_{\compensators_\cli}^T \bar d \right) \right]
-\sin \theta \cos\theta  U_N^2
\end{align*}
and
\begin{align*}
\left[\tilde K_r(\uc)\right]_k & = 
\cos\theta\Im\left[e^{-j\theta} d_k \right]\\
&\quad+
\frac{\cos\theta}{U_N |\compensators_\cli|}\Im \left[e^{-j\theta}
\left(\1_{\compensators_\cli}^T X s \right)\left[ X \bar s \right]_k \right]\\
&\quad+
\frac{\cos\theta}{U_N^2 |\compensators_\cli|}\Im \left[e^{-j2\theta}
\left(\1_{\compensators_\cli}^T X s \right)d_k\right]\\
&\quad+\frac{\cos\theta}{U_N^2 |\compensators_\cli|}\Im \left[\left(\1_{\compensators_\cli}^T \bar d \right)
\left[ X \bar s \right]_k \right]\\
&\quad+
\frac{\cos\theta}{U_N^3 |\compensators_\cli|}\Im \left[e^{-j\theta}
\left(\1_{\compensators_\cli}^T \bar d \right) d_k\right],
\end{align*}
which according to Proposition \ref{proposition:taylor_expansion_complex} is bounded when $U_N\to\infty$.
\end{proof}

\subsection{Description of the algorithm}

By applying Proposition \ref{pro:Kdecomp} to expression \eqref{eq:node_law} for 
$\qopti$ we obtain that, for $h \in \compensators_\cli$, 
\begin{equation}
\begin{split}
\left[\qopti\right]_h &=
	q_h 
-
\cos \theta 
\sum_{k\in\mathcal C_\cli}
\left[\left(\Omega_\cli M \Omega_\cli \right)^\sharp\right]_{hk}
\left[K_r(\uc)\right]_k \\
&\quad -
\frac{\cos\theta}{U_N} 
\sum_{k\in\mathcal C_\cli}
\left[\left(\Omega_\cli M \Omega_\cli \right)^\sharp\right]_{hk}
\left[\tilde K_r(\uc)\right]_k,
\end{split}
\label{qott-appr1}
\end{equation}
where the term $\alpha_r$ is canceled because of the specific kernel of $(\Omega_r M \Omega_r)^\sharp$.
Because $\frac{1}{U_N}\tilde K_r$ is infinitesimal as $U_N\to\infty$, then equation (\ref{qott-appr1}), 
together with identity \eqref{eq:HdaRsimpl},
suggests the following approximation for $\qopti$ 
\begin{equation}\label{qott-appr111}
\left[\widehat{\qopti}\right]_h=
	q_h + 2 \cos \theta \sum_{k\in\mathcal C_\cli}
\left[\left(\Omega_\cli R^\text{eff} \Omega_\cli \right)^\sharp\right]_{hk}
\left[K_r(\uc)\right]_k
\end{equation}
which can be computed for all $h \in \compensators_\cli$ by the supervisor of cluster $r$ from local information, 
namely the local network electric properties and the voltage measurements at the compensators belonging to $\compensators_\cli$.
Notice also that, 
as $\Imag(\Omega_\cli R^\text{eff} \Omega_\cli)^\sharp = \Imag \Omega_\cli$, 
we have that $\sum_h \left[ \qopti \right]_h - q_h = 0$, 
and therefore the constraint $\1^T q = 0$ remains satisfied.

Based on this result, we propose the following iterative algorithm,
whose online part is also depicted in Figure~\ref{fig:algorithm} in its block diagram representation.

\noindent\hrulefill
\textsc{ Offline procedure }
\noindent\hrulefill
\begin{itemize}
\item The supervisor of each cluster $\compensators_\cli$ gathers
the system parameter $[R^\text{eff}]_{hk}$, $h,k\in \compensators_\cli$;
\item each supervisor computes the elements $\left[\left(\Omega_\cli R^\text{eff} \Omega_\cli\right)^\sharp\right]_{hk}$, $h,k \in \compensators_\cli$.
\end{itemize}

\noindent\hrulefill
\textsc{ Online procedure }
\noindent\hrulefill\\
At each iteration $t$:
\begin{itemize}
\item a cluster $\compensators_\cli$ is randomly selected among all the clusters;
\item each compensator $h \in \compensators_\cli$ measures its voltage $u_h$;
\item each compensator $h \in \compensators_\cli$ sends its voltage $u_h$
and its state $q_h$ (the amount of reactive power they are injecting into the grid)
to the cluster supervisor;
\item the cluster supervisor computes $K_r(\uc)$, via \eqref{eq:def_K};
\item the cluster supervisor computes $\left[ \widehat \qopti \right]_h$ for all $h \in \compensators_\cli$ via \eqref{qott-appr111};
\item each compensator $h \in \compensators_\cli$ receives the value $\left[ \widehat \qopti \right]_h$ from the cluster supervisor and updates its injection of reactive power into the grid (i.e. $q_h$) to that value.
\end{itemize}
\noindent\hrulefill

\begin{figure}[tb]
\centering
\includegraphics[width=0.95\IEEECOLWIDTH]{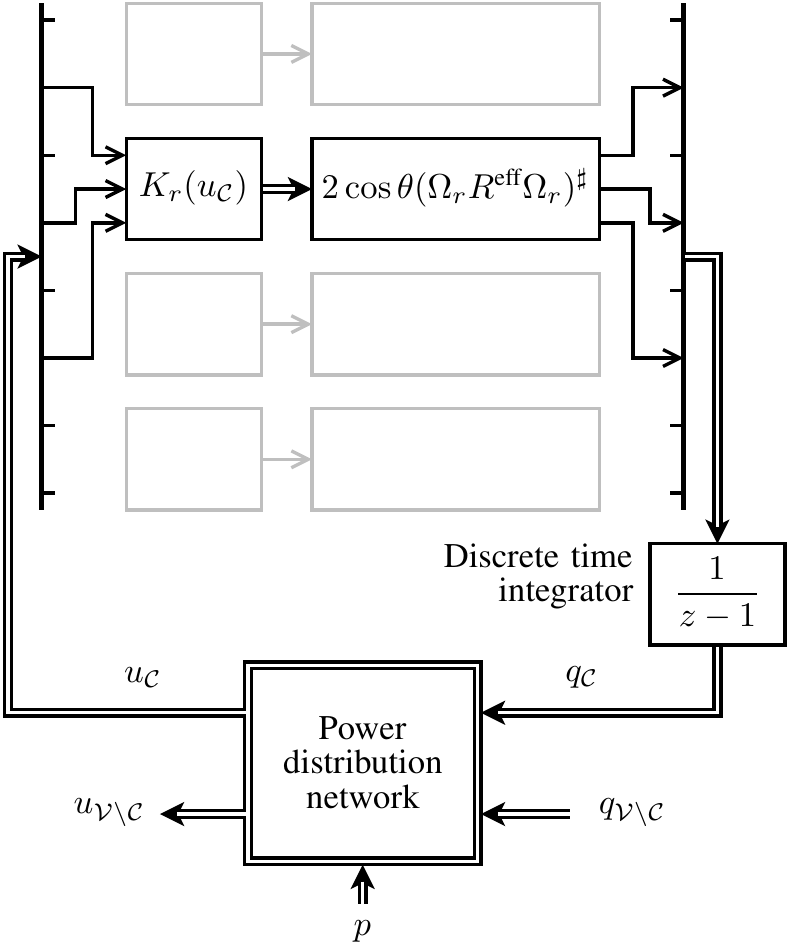}
\caption{Block diagram representation of the proposed distributed algorithm. The repeated solution of different optimization subproblems corresponds to a switching discrete time feedback system. Double arrows represent vector signals. The \emph{power distribution network} block receives, as inputs, the active and reactive power injection of both compensators (nodes in $\compensators$) and loads (nodes in $\nodes \backslash \compensators$). It gives, as outputs, the voltages at the nodes. The voltages at the compensators ($u_\compensators$) can be measured, and provide the feedback signal for the control laws. Notice that the control law corresponding to a specific cluster $r$ requires only the voltage measurements performed by the compensators belonging to the cluster ($u_k, k \in \compensators_\cli$), and it updates only the reactive power injected by the same compensators ($q_k, k \in \compensators_\cli$), while the other compensators hold their reactive power injection constant.}
\label{fig:algorithm}
\end{figure}

\begin{remark}
One possible way of randomly selecting one cluster $\mathcal C_\cli$ at each iteration, is the following. 
Let each cluster supervisor be provided with a timer, each one governed by an independent Poisson process, which triggers the cluster after exponentially distributed waiting times.
If the time required for the execution of the algorithm is negligible with respect to the typical waiting time of the Poisson processes, this strategy allows to construct the sequence of independent symbols $\cli(t)$, $t=0,1,2,\ldots$, in a purely decentralized way (see for example \cite{Boyd2006}).
\end{remark}

\begin{remark}
It is important to notice that the distributed implementation of the algorithm is only possible because of 
the specific physical system that we are considering. 
By sensing the system \emph{locally}, nodes can infer some \emph{global} information on the system state (namely, the gradient of the cost function) that otherwise would depend on the states of all other nodes.
For this reason, it is necessary that, after each iteration of the algorithm, the microgenerators actuate the system implementing the optimization step, so that subsequent measurement will reflect the state change.
This approach resembles some other applications of distributed optimization,
e.g. the radio transmission power control in wireless networks \cite{Grandhi_1994_Distributedpowercontrol} and the traffic congestion protocol for wired data networks \cite{Kelly_1998_Ratecontrolcommunication,Low_2003_dualitymodelof}.
Also in these cases, the iterative tuning of the decision variables (radio power, transmission rate) depends on 
congestion feedback signals that are function of the entire state of the system.
However, these signals can be detected locally by each agent by measuring error rates, signal-to-noise ratios, or specific feedback signals in the protocol.
\end{remark}


\section{Convergence of the algorithm}
\label{sec:convergence}

In this section we will give the conditions ensuring that the proposed iterative algorithm converges to the optimal solution, and we will analyze its convergence rate.  In fact, according to \eqref{qott-appr111}, the dynamics of the algorithm is described by the iteration  
\begin{equation}
q_{\mathcal C}(t+1)
=
q_{\mathcal C}(t) -
\cos \theta (\Omega_{r(t)}M \Omega_{r(t)})^\sharp K_{r(t)}(\uc(t)).
\label{eq:updateStepreal}
\end{equation}
where the voltages $\uc(t)$ are a function of the reactive power $q_{\mathcal C}(t)$, $q_{\mathcal V\setminus\mathcal C}$ and of the active power $p$.
We will consider instead the iteration
\begin{equation}
q_{\mathcal C}(t+1)
=
q_{\mathcal C}(t)
-
\left( \Omega_{r(t)} M \Omega_{r(t)} \right)^\sharp (Mq_{\mathcal C}(t)+Nq_{\mathcal V\setminus\mathcal C}),
\label{eq:updateStep}
\end{equation}
which descends from \eqref{eq:argmin23} and which 
differs from \eqref{eq:updateStepreal} for the term
$$
\frac{\cos \theta}{U_N}\left( \Omega_{r(t)} M \Omega_{r(t)} \right)^\sharp\tilde K_{r(t)}(\uc),
$$
as one can verify by inspecting \eqref{qott-appr1}.
This difference is infinitesimal as $U_N\to\infty$.
Therefore in the sequel we will study the convergence of (\ref{eq:updateStep}), while
numerical simulations in Section~\ref{sec:simulations} will support the proposed approximation, showing how both the algorithm steady state performance and the rate of convergence are practically the same for the exact system and the approximated model.

\begin{remark}
A formal proof of the convergence of the discrete time system (\ref{eq:updateStepreal}) could be done based of the mathematical tools proposed in Sections 9 and 10 of \cite{khalil} dedicated to the stability of perturbed systems. However, in our setup there are a number of complications that make the application of those techniques not easy. The first is that, since $u_{\mathcal C}(t)$ is related to $q_{\mathcal C}(t)$ in an implicit way, then (\ref{eq:updateStepreal}) is in fact a nonlinear implicit (also called descriptor) system with state $q_{\mathcal C}(t)$. The second difficulty comes from the fact that (\ref{eq:updateStepreal}) is a nonlinear randomly time-varying (also called switching) system. 
\end{remark}

In order to study the system (\ref{eq:updateStep}),
we introduce the auxiliary variable 
$x = q_{\mathcal C} - q_{\mathcal C}^\text{opt} \in \realnumbers^m$,
where $q^\text{opt}$  is the solution of the optimization problem \eqref{eq:Qoptimizationproblem},
and $q_{\mathcal C}^\text{opt}$ are the elements of $q^\text{opt}$ corresponding to the nodes in $\compensators$.
In this notation, it is possible to explicitly express the iteration
\eqref{eq:updateStep}
as a discrete time system in the form%
\begin{equation}
x(t+1) = F_{\cli(t)} x(t), \quad x(0)\in\ker \1^T,
\label{eq:linearUpdateLaw}
\end{equation}
where
\begin{equation}
F_\cli = I - (\Omega_\cli M \Omega_\cli)^\sharp M.
\label{eq:Fdefinition}
\end{equation}
The matrices $F_\cli$ have some nice properties which are listed below.
\begin{enumerate}[i)]
\item The matrices $F_\cli$ are self-adjoint matrices with respect to the inner product $\langle \cdot, \cdot \rangle_M$, defined as $\langle x, y \rangle_M := x^T M y$.
Therefore $F_\cli^T M = M F_\cli$, and thus they have real eigenvalues.
\item The matrices $F_\cli$ are projection operators, i.e. $F_\cli^2 = F_\cli$, 
and they are orthogonal projections with respect to the inner product $\langle \cdot, \cdot \rangle_M$, i.e. $\langle F_\cli x , F_\cli x - x \rangle_M =  x^T F_\cli^T M (F_\cli x - x) = 0$.
\end{enumerate}
Notice that, as $\Imag(\Omega_\cli M \Omega_\cli)^\sharp = \Imag \Omega_\cli$, 
$\ker \1^T$ is a positively invariant set for \eqref{eq:linearUpdateLaw}, namely $F_\cli(\ker \1^T)\subseteq \ker \1^T$ for all $\cli$.

It is clear that $q(t)$ converges to the optimal solution $q^\text{opt}$ if and only if 
$x(t)$ converges to zero for any initial condition $x(0)\in\ker \1^T$. 
A necessary condition for this to happen 
is that there are no nonzero equilibria in the discrete time system \eqref{eq:linearUpdateLaw},
namely that the only $x\in\ker \1^T$ such that $F_\cli x=x$ for all $\cli$ must be $x=0$. 

The following proposition provides a convenient characterization of this property,
in the form of a connectivity test.
Let ${\cal H}$ be the hypergraph whose nodes are the compensators, and whose edges are the clusters $\mathcal C_\cli$, $\cli=1,\ldots,\ell$. This hypergraph describes the communication network needed for implementing the algorithm: compensators belonging to the same edge (i.e. cluster) must be able to share their measurements with a supervisor and receive commands from it.

\begin{proposition}
\label{pro:connectedHypergraph}
Consider the matrices $F_\cli$, $\cli=1,\ldots,\ell$, defined in \eqref{eq:Fdefinition}. 
The point $x = 0$ is the only point in $\ker \1^T$ such that  $F_\cli x=x$ for all $\cli$
if and only if the hypergraph ${\cal H}$ is connected.
\end{proposition}

\begin{proof}
The proposition is part of the statement of Lemma~\ref{lem:threeways} in the Appendix~\ref{app:convergence}.
\end{proof}

In order to prove that the connectivity of the hypergraph ${\mathcal H}$ is not only a necessary but also a sufficient condition for the convergence of the algorithm, we introduce the following assumption on the sequence $\cli(t)$.

\begin{assumption}
\label{ass:randomSubproblemExecution}
The sequence $\cli(t)$ is a sequence of independently, identically distributed symbols in $\{1,\ldots, \ell\}$,
with non-zero probabilities $\{\rho_\cli > 0, \cli = 1,\ldots, \ell\}$.
\end{assumption}

Let  $v(t) := \expect{x^T(t)Mx(t))} = \expect{J(q(t)) - J(q^\text{opt})}$ and consider the following performance metric
\begin{equation}
R = \sup_{x(0) \in \ker \1^T} \lim\sup \{v(t)\}^{1/t}
\label{eq:R}
\end{equation}
which describes the exponential rate of convergence to zero of $v(t)$. 
It is clear that $R<1$ implies the exponential mean square convergence of 
$q_\compensators(t)$ to the optimal solution $q_\compensators^\text{opt}$.

Using \eqref{eq:linearUpdateLaw}, we have
\begin{equation*}
\begin{split}
v(t) &= \expect{x(t)^T M x(t)} 
= \expect{x(t)^T \Omega M \Omega x(t)} \\
&= \expect{x(t-1)^T F_{\cli(t-1)}^T \Omega M \Omega F_{\cli(t-1)} x(t-1)} \\
&= x(0)^T \expect{F_{\cli(0)}^T \cdots F_{\cli(t\text{--}1)}^T \Omega M \Omega F_{\cli(t\text{--}1)} \cdots F_{\cli(0)}} x(0). \\
\end{split}
\end{equation*}
Let us then define
$$
\Delta(t) = \expect{F_{\cli(0)}^T \cdots F_{\cli(t-1)}^T \Omega M \Omega F_{\cli(t-1)} \cdots F_{\cli(0)}}.
$$

Via Assumption \ref{ass:randomSubproblemExecution}, 
we can argue that $\Delta(t)$ satisfies the following linear recursive equation 
\begin{equation}
\begin{cases}
\Delta(t+1) =   \mathcal L(\Delta(t))  \\
\Delta(0) = \Omega M \Omega,
\end{cases}
\label{eq:deltaSystem}
\end{equation}
where $\mathcal L(\Delta):=\expect{F_\cli^T \Delta F_\cli}$.
Moreover we have that
\begin{equation}
v(t) = x(0)^T \Delta(t) x(0).
\label{eq:deltaSystem1}
\end{equation}
Equations \eqref{eq:deltaSystem} and \eqref{eq:deltaSystem1} can be seen as a discrete time linear system with state $\Delta(t)$ and output $v(t)$.

\begin{remark}
The discrete time linear system \eqref{eq:deltaSystem} can be equivalently be described by the equation
$$
\vect\left( \Delta(t+1) \right) = \mathbf F \vect\left( \Delta(t) \right),
$$
where $\vect(\cdot)$ stands for the operation of vectorization and where
$$
\mathbf F := \expect{F_\cli^T \otimes F_\cli^T}\in\mathbb R^{\nocompensators^2 \times \nocompensators^2}.
$$
Notice that $\mathbf F $ is self-adjoint with respect to the inner product $\langle \cdot, \cdot \rangle_{M^{-1} \otimes M^{-1}}$ and so $\mathbf F$, and consequently $\mathcal L$, has real eigenvalues.
\end{remark}

Studying the rate $R$ (which can be proved to be the slowest reachable and observable mode of 
the system \eqref{eq:deltaSystem} with the output \eqref{eq:deltaSystem1}, see \cite{Bolognani_2011_GossiplikeDistributedOptimization})
is in general not simple. It has been done analytically and numerically for some special graph structures in
\cite{Bolognani_2011_GossiplikeDistributedOptimization}.
It is convenient to analyze its behavior indirectly, 
through another parameter which is easier to compute and to study,
as the following result shows.

\begin{theorem}
Assume that Assumption \ref{ass:randomSubproblemExecution} holds true and that the hypergraph ${\cal H}$ is connected.
Then the rate of convergence $R$, defined in \eqref{eq:R}, satisfy
$$R \le \beta<1,$$
where
\begin{equation*}
\beta = \max \{ |\lambda| \;|\; \lambda \in \lambda(\Fave), \lambda \ne 1\},
\end{equation*}
with $\Fave := \expect{F_\cli}$.
\label{thm:betaLessThanR}
\end{theorem}
\begin{proof}
The proof of Theorem~\ref{thm:betaLessThanR} is given in the Appendix~\ref{app:convergence}.
\end{proof}

The following result descends directly from it.

\begin{corollary}
Assume that Assumption \ref{ass:randomSubproblemExecution} holds true and that the hypergraph ${\cal H}$ is connected.
Then the state of the iterative algorithm described in Section \ref{subsec:problemdecomposition} converges in mean square to the global optimal solution.
\end{corollary}

\begin{remark}
Observe that, by standard arguments based on Borel Cantelli lemma, the exponential convergence in mean square implies also almost sure convergence to the optimal solution.
\end{remark}

The tightness of $\beta$ as a bound for $R$ has been studied in \cite{Bolognani_2011_GossiplikeDistributedOptimization},
by evaluating both $\beta$ and $R$ analytically and numerically for some special graph topologies.
In the following we consider $\beta$ as a reliable metric for the evaluation of the algorithm performances.

The following result shows what is the best performance 
(according to the bound $\beta$ on the convergence rate $R$)
that the proposed algorithm can achieve.

\begin{theorem}
Consider the algorithm \eqref{eq:updateStep},
and assume that $\mathcal H$ describing the clusters $\mathcal C_\cli$
is an arbitrary connected hypergraph defined over the nodes $\compensators$. 
Let Assumption \ref{ass:randomSubproblemExecution} hold.
Then 
\begin{equation}\label{eq:boundbetagen}
\beta \ge 1 - \frac{\left(\sum_{\cli=1}^\ell \rho_\cli |\mathcal C_\cli|\right)-1}{m - 1}.
\end{equation}
In case all the sets $\mathcal C_\cli$ have the same cardinality $c$, namely $|\mathcal C_\cli|=c$ for all $\cli$, then
\begin{equation}\label{eq:boundbetasimp}
\beta \ge 1 - \frac{c-1}{m - 1}.
\end{equation}
\label{thm:bestbeta}
\end{theorem}

\begin{proof}
Let 
\begin{align}
E_\cli &:= (\Omega_\cli M \Omega_\cli)^\sharp M  \label{def:Ei} \\
\Eave &:= \expect{E_\cli}=\sum_{\cli=1}^\ell \rho_\cli E_\cli\label{def:barE}
\end{align}
It is clear that $\beta = 1 - \beta'$,
where $\beta'  = \min \{ |\lambda| \;|\; \lambda \in \lambda(\Eave), \lambda \ne 0\}$.
We have that
$$
\sum_{\lambda_j\in \lambda(\Eave)}\!\! \lambda_j = 
\trace(\Eave) =
\trace\left( \sum_{\cli=1}^\noclusters \rho_\cli E_\cli \right)=
\sum_{\cli=1}^\noclusters \rho_\cli \trace \left( E_\cli \right).
$$
Notice now that $E_\cli$ have eigenvalues zero or one and so 
\begin{align*}
\trace\left( E_\cli \right)
&=\text{rank } E_\cli= \dim \left(\Imag E_\cli\right)\\
&=\dim\left(\Imag (\Omega_\cli M \Omega_\cli)^\sharp\right)
=\dim\left(\Imag \Omega_\cli \right)\\
&=|\mathcal C_\cli|-1.
\end{align*}
This implies that
$$
\sum_{\lambda_j\in \lambda(\Eave)} \lambda_j
=
\left(\sum_{\cli=1}^\noclusters \rho_\cli|\mathcal C_\cli|\right)-1,
$$
and so
\begin{align*}
\beta'  &= \min \{ |\lambda| \;|\; \lambda \in \lambda(\Eave), \lambda \ne 0\}\\
&\le
\frac{\sum_{\lambda_j\in \lambda(\Eave)} \lambda_j}{\nocompensators - 1}
=
\frac{\left(\sum_{\cli=1}^\noclusters \rho_\cli|\mathcal C_\cli|\right)-1}{\nocompensators-1}.
\end{align*}
\end{proof}

\section{Optimal communication hypergraph for a radial distribution network}
\label{sec:optimalHypergraph}

In this section we present a special case in which the optimal convergence rate of Theorem \ref{thm:bestbeta} is indeed achieved via a specific choice of the clusters $\mathcal C_\cli$. 
To obtain this result we need to introduce the following assumption, 
which is commonly verified in many practical cases (including the vast majority of power distribution networks and the standard IEEE testbeds \cite{Kersting2001}). 

\begin{assumption}
\label{ass:tree}
The distribution network is radial, i.e. the corresponding graph $\graph$ is a tree. 
\end{assumption}

We start by providing some useful properties of the matrices $E_\cli$.

\begin{lemma}
The following properties hold true:
\begin{enumerate}[i)]
\item $\Imag E_\cli=\Imag\Omega_\cli$;
\item $E_\cli^2=E_\cli$;
\item $E_\cli x=x$ for all $x\in\Imag\Omega_\cli$.
\end{enumerate}
\label{lemma:propEi}
\end{lemma}
\begin{proof}
\begin{enumerate}[i)]
\item First observe that, since $M$ is invertible, then  $\Imag E_\cli=\Imag(\Omega_\cli M \Omega_\cli)^\sharp M=\Imag(\Omega_\cli M \Omega_\cli)^\sharp$. 
Finally from the properties of the matrices $(\Omega_\cli M \Omega_\cli)^\sharp$, 
we obtain that $\Imag E_\cli=\Imag\Omega_\cli$.
\item First observe that $E_\cli^2=(I-F_\cli)^2=I-2F_\cli+F_\cli^2$. Now, since $F_\cli^2=F_\cli$ we obtain that $I-2F_\cli+F_\cli^2=I-F_\cli=E_\cli$.
\item Let $x\in\Imag\Omega_\cli=\Imag E_\cli$. Then $x=E_\cli z$ for some vector $z$ and so $E_\cli x=E_\cli^2 z=E_\cli z=x$.
\end{enumerate}
\end{proof}

Let us now consider for any pair of nodes $h,k$ the shortest path  $\mathcal P_{hk} \subseteq \edges$
(where, we recall, a path is a sequence of consecutive edges) connecting $h,k$.
Define moreover
$$
\mathcal P_{\mathcal C_\cli}:=\bigcup_{h,k\in\mathcal C_\cli}\mathcal P_{hk}.
$$

\begin{lemma}
Let Assumption \ref{ass:tree} hold true. Then
\begin{enumerate}[i)]
\item $(\1_h-\1_k)^TM(\1_{h'}-\1_{k'})=\sum_{e\in\mathcal P_{hk}\cap \mathcal P_{h'k'}}\Re(z_e)$.
\item $E_\cli (\1_h-\1_k)=0$ if $\mathcal P_{\mathcal C_\cli}\cap \mathcal P_{hk}=\emptyset$.
\end{enumerate}
\label{lemma:tree_M}
\end{lemma}

\begin{proof}
\begin{enumerate}[i)]
\item Observe that $(\1_h-\1_k)^TM(\1_{h'}-\1_{k'})=(\1_h-\1_k)^TX(\1_{h'}-\1_{k'})$,
where with some abuse of notation we have denoted with the same symbol $\1_h$ the vector in $\realnumbers^\nocompensators$ and the corresponding vector in $\realnumbers^\nonodes$.
Notice that, if we have a current $\cli=\1_{h'}-\1_{k'}$ in the network, 
then according to \eqref{eq:det_iu_U0} we get a voltage vector
$$
u= \green i+\1 U_N= \green (\1_{h'}-\1_{k'})+\1 U_N
$$
and so $(\1_h-\1_k)^T \green (\1_{h'}-\1_{k'})=u_h-u_k$. Observing the following figure

\begin{center}
\bigskip
\includegraphics[scale=0.6]{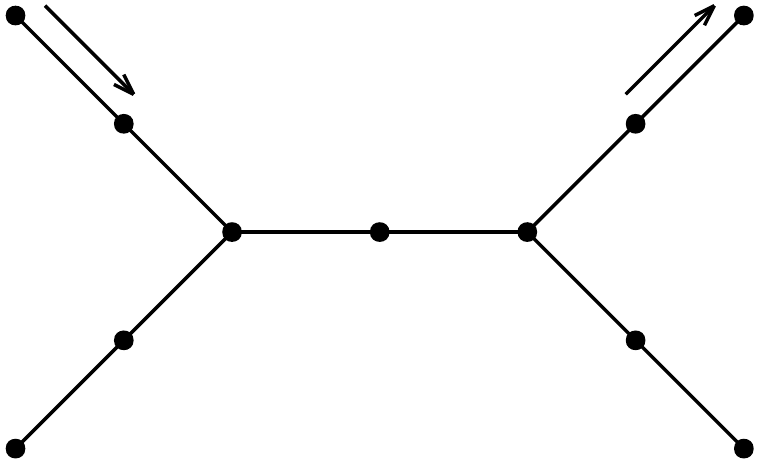}
\put(-138,76){\small $h'$}
\put(3,76){\small $k'$}
\put(3,0){\small $k$}
\put(-140,0){\small $h$}
\put(-93,46){\small $h''$}
\put(-44,46){\small $k''$}
\end{center}

we can argue that $u_h-u_k=u_{h''}-u_{k''}=\sum_{e\in\mathcal P_{h''k''}}z_e$ 
and so
$(\1_h-\1_k)^T M (\1_{h'}-\1_{k'})=\sum_{e\in\mathcal P_{h''k''}}\Re(z_e)$.
\item If $\mathcal P_{\mathcal C_\cli}\cap \mathcal P_{hk}=\emptyset$, then
\begin{multline*}
$$\Omega_\cli M(\1_h-\1_k)
=\\
\frac{1}{2|\mathcal C_\cli|}\sum_{h',k'\in\mathcal C_\cli}(\1_{h'}-\1_{k'})(\1_{h'}-\1_{k'})^TM(\1_{h'}-\1_{k'})=0
$$
\end{multline*}
since $\mathcal P_{hk}\cap \mathcal P_{h'k'}=\emptyset$ for all $h',k' \in \compensators_\cli$. 
Thus $E_\cli (\1_h-\1_k)=(\Omega_\cli M \Omega_\cli)^\sharp\Omega_\cli M(\1_h-\1_k)=0$.
\end{enumerate}
\end{proof}

We give now the key definition to characterize, together with connectivity, the optimality of the communication hypergraph $\mathcal H$.

\begin{definition}
\label{def:minimality}
The hyperedges $\{\mathcal C_\cli, \,\cli = 1,\ldots,\ell\}$ of the hypergraph $\mathcal H$ are \emph{edge-disjoint}\footnotemark\ if for any $\cli,\cli'$ such that $\cli\ne \cli'$,
we have that
$$
\mathcal P_{\mathcal C_\cli} \cap \mathcal P_{\mathcal C_{\cli'}}=\emptyset.
$$
\end{definition}

\footnotetext{This definition of \emph{edge-disjoint} for this specific scenario resembles a similar concept for routing problems in data networks, as for example \cite{Srinivasan_1997_edgedisjoint}.}

In order to prove the optimality of having edge-disjoint clusters,
we need the following lemma.

\begin{lemma}
Assume that Assumptions \ref{ass:tree} holds true and that the clusters $\compensators_\cli$ are edge-disjoint.
Then $E_\cli E_{\cli'}=0$ for all $\cli \ne \cli'$.
\label{lemma:orth}
\end{lemma}

\begin{proof} Let $x$ be any vector in $\realnumbers^\nocompensators$. 
Notice that, since $\Imag E_{\cli'}=\Imag \Omega_{\cli'}$, then there exists $z\in\mathbb R^m$ such that
$$
E_\cli E_{\cli'} x = E_\cli \Omega_{\cli'} z
=
\frac{1}{2|\mathcal C_\cli|}\sum_{h,k\in\mathcal C_{\cli'}}
\!\!\!\!\!E_\cli(\1_h-\1_k)(\1_h-\1_k)^Tz=0
$$
where the last equality follows from the fact that 
$\mathcal P_{\mathcal C_\cli}\cap \mathcal P_{hk}=\emptyset$ 
for all $h,k\in\mathcal C_{\cli'}$.
\end{proof}

From the previous lemma we can argue that every vector in $\Imag\Omega_\cli$ is eigenvector of 
$\Eave = \sum_{\cli=1}^\ell \rho_\cli E_\cli$ associated with the eigenvalue $\rho_\cli$. 
On the other hand, recall that the hypergraph connectivity implies that $\ker\1^T=\sum_{\cli=1}^\ell\Imag\Omega_\cli=\sum_{\cli=1}^\noclusters \Imag E_\cli$ and hence we have that,
besides the zero eigenvalue, the eigenvalues of  $\Eave$ are exactly $\rho_1,\ldots,\rho_\noclusters$
and so 
$$
\beta = 1 - \min_{\cli=1,\ldots,\ell}\{\rho_\cli\}.
$$
The bound $\beta$ is minimized by taking $\rho_\cli=1/\ell$ for all $\cli$, obtaining
$$
\beta=1-\frac{1}{\ell}.
$$
Notice finally that $E_\cli E_{\cli'}=0$ for all $\cli \ne \cli'$ 
implies that $\sum_{\cli=1}^\noclusters \Imag E_\cli$ 
is a direct sum and then
$$
\nocompensators-1 = \sum_{\cli=1}^\noclusters \dim \Imag E_\cli
= \sum_{\cli=1}^\noclusters(|\mathcal C_\cli|-1)
=\sum_{\cli=1}^\noclusters |\mathcal C_\cli| - \noclusters.
$$
If all the sets $\mathcal C_\cli$ have the same cardinality $c$, then $\ell c = m + \ell - 1$,
and so
$$
\beta = 1-\frac{1}{\ell} = 1-\frac{c-1}{m-1}
$$
which, according to Theorem \ref{thm:bestbeta}, shows the optimality of the hypergraph.

%

\section{Simulations}
\label{sec:simulations}

\begin{figure}[tb]
\centering
\includegraphics[scale=0.5]{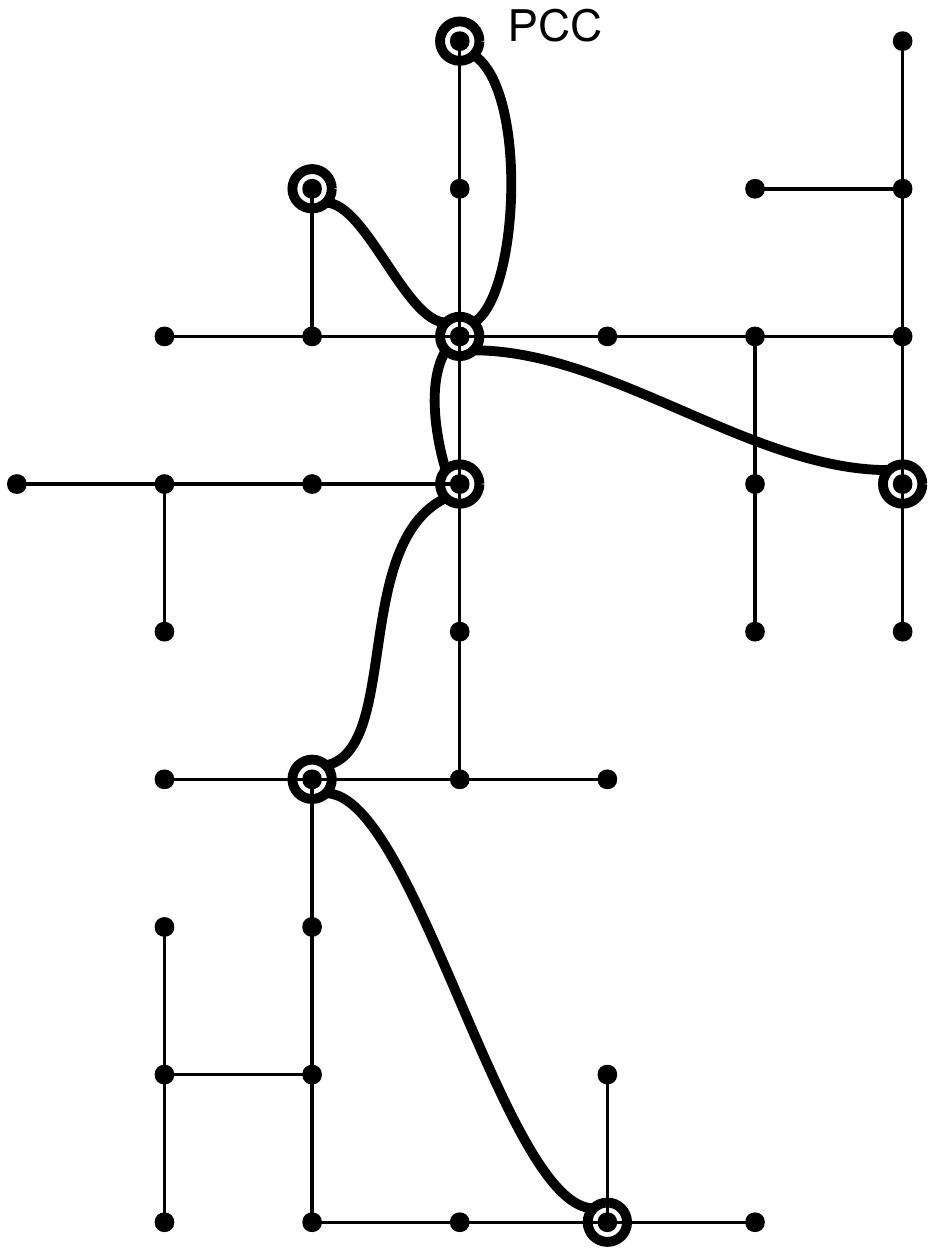}
\caption{Schematic representation of the IEEE37 testbed. 
Circled nodes represent microgenerators taking part to the distributed reactive power compensation.
The thick curved lines represent clusters of size $c=2$ (i.e. edges of $\mathcal H$)
connecting pair of compensators that can communicate and coordinate their behavior.}
\label{fig:ieee37}
\end{figure}

In this section we present numerical simulations to validate 
both the model presented in Section~\ref{sec:problem_formulation},
and the randomized algorithm proposed in Section~\ref{sec:randomizedAlgorithm}.

To do so, we considered a 4.8 kV testbed inspired from the standard testbed IEEE 37 \cite{Kersting2001},
which is an actual portion of power distribution network located in California.
We however assumed that load are balanced, and therefore all currents and voltages can be described in a single-phase phasorial notation.

As shown in Figure \ref{fig:ieee37}, 
some of the nodes are microgenerators
and are capable of injecting a commanded amount of reactive power.
Notice that the PCC (node $\PCC$) belongs to the set of compensators.
This means that the microgrid is allowed to change the amount of reactive power 
gathered from the transmission grid, if this reduces the power distribution losses.
The nodes which are not compensators, are a blend of constant-power, constant-current, and constant-impedance loads,
with a total power demand of almost 2 MW of active power and 1 MVAR of reactive power
(see \cite{Kersting2001} for the testbed data).

The network topology is radial. While this is not a necessary condition for the algorithm, this is typical in basically all power distribution networks. Moreover, the choice of a radial grid allows the validation of the optimality result of Section \ref{sec:optimalHypergraph} for the clustering of the compensators.

The length of the power lines range from a minimum of 25 meters to a maximum of almost 600 meters.
The impedance of the power lines differs from edge to edge 
(for example, resistance ranges from 0.182 $\Omega$/km to 1.305 $\Omega$/km).
However, the inductance/resistance ratio exhibits a smaller variation,
ranging from $\angle z_e = 0.47$ to $\angle z_e = 0.59$.
This justifies Assumption \ref{ass:theta}, in which we claimed that $\angle z_e$ can be considered constant across the network. The effects of this approximation will be commented later.

Without any distributed reactive power compensation,
distribution power losses amount
to 61.6~kW, 3.11\% of the delivered active power.

Given this setup, we first estimate the quality of the linear approximated model 
proposed in Section~\ref{sec:problem_formulation}.
As shown in Figure~\ref{fig:voltages}, 
the approximation error results to be negligible.

\begin{figure}[tb]
\raggedleft
\includegraphics[scale=0.94]{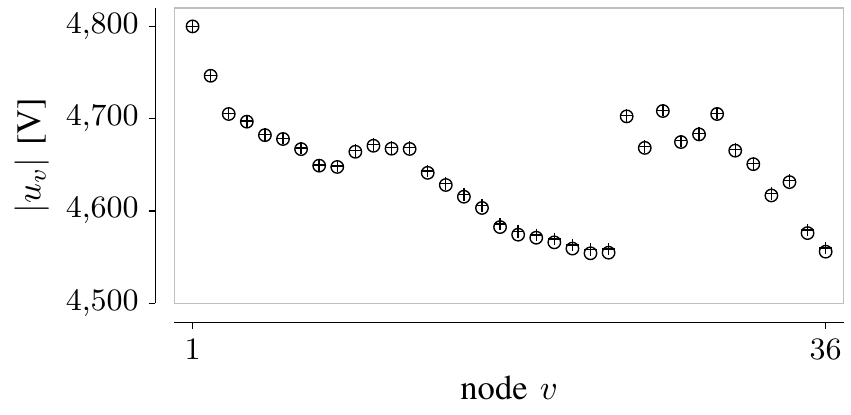}$\ $\\
\includegraphics[scale=0.94]{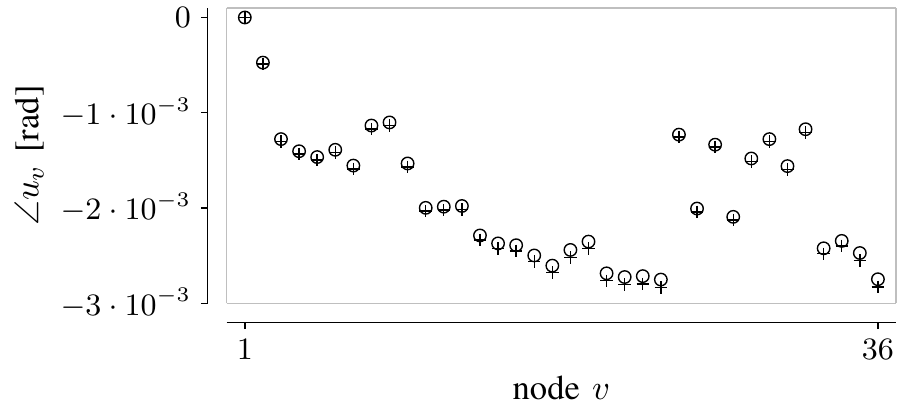}$\ $ 
\caption{Comparison between the network state (node voltages) 
computed via the exact model induced by \eqref{eq:KCL}, \eqref{eq:KVL}, 
\eqref{eq:ZIPModel}, and \eqref{eq:PCCmodel} ($\circ$), 
and the approximate model proposed in Proposition \ref{proposition:taylor_expansion_complex} ($+$).}
\label{fig:voltages}
\end{figure}

We then simulated the behavior of the algorithm proposed in Section \ref{sec:randomizedAlgorithm}.
We considered the two following clustering choices:
\begin{itemize}
\item \emph{edge-disjoint gossip:} motivated by the result stated in Section \ref{sec:optimalHypergraph},
we enabled pairwise communication between compensator in a way that guarantees that the clusters (pairs) are edge-disjoint; the resulting hypergraph $\mathcal H$ is represented as a thick line in Figure \ref{fig:ieee37};
\item \emph{star topology:} clusters are in the form $\mathcal C_\cli = \{\PCC, v\}$ for all $v \in \mathcal C \backslash \{\PCC\}$; the reason of this choice is that, as $\PCC$ is the PCC, 
the constraint $\1^T q = 0$ is inherently satisfied:
whatever variation in the injected reactive power is applied by $v$, 
it is automatically compensated by a variation in the demand of reactive power from the transmission grid via the PCC.
\end{itemize}

\begin{figure}[tb]
\raggedleft
\includegraphics[scale=0.94]{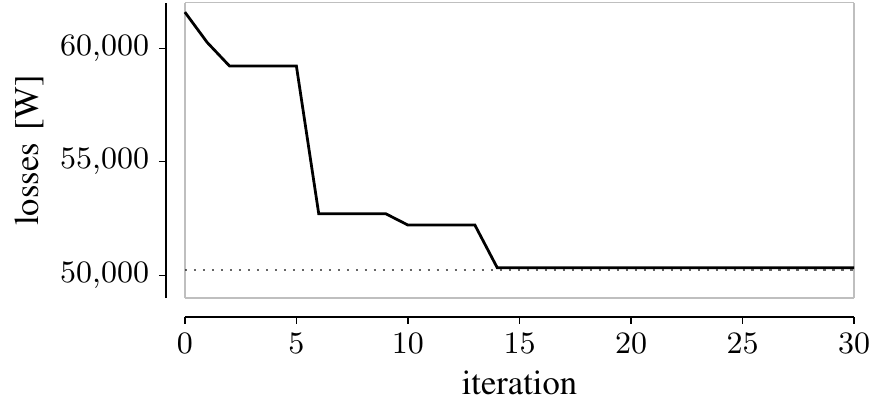}$\ $%
\caption{Power distribution losses resulting by the execution of the proposed algorithm.
A edge-disjoint hypergraph, yielding optimal convergence speed, has been adopted.
The dashed line represent the minimum losses that can be achieved via centralized numerical optimization.}
\label{fig:singlerun}
\end{figure}

In Figure \ref{fig:singlerun} we plotted the result of a single execution of the algorithm in the \emph{edge-disjoint gossip} case. 
One can see that the algorithm converges quite fast, reducing losses to a minimum that is extremely close to the best achievable solution. The results achieved by the proposed algorithm on this testbed are summarized in the following table.

\begin{center}
\begin{tabular}{rr}
\hline 
Losses before optimization & 61589 W \\
Fraction of delivered power & 3.11 \% \\
\hline
Losses after optimization & 50338 W \\
Fraction of delivered power & 2.55 \% \\
losses reduction	& 18.27 \% \\
\hline
Minimum losses $J^\text{opt}$ & 50253 W \\
Fraction of delivered power & 2.54 \% \\
losses reduction & 18.41 \% \\
\hline
\end{tabular}
\end{center}

The minimum losses $J^\text{opt}$ 
is the solution of the original 
optimization problem \eqref{eq:nl_optimization_problem}, 
has been obtained by a centralized numerical solver, 
and represent the minimum losses that can be achieved by properly choosing the amount of reactive power injected by the compensators (and retrieved from the PCC).
The difference between this minimum and the minimum achieved by the algorithm proposed in this paper 
is partly due to the approximation that we introduced when we modeled the microgrid (assuming large nominal voltage $U_N$),
and partly due to the assumption that $\theta$ is constant across the network.
Further simulative investigation on this testbed 
showed that the effect of this last assumption is largely predominant, 
compared to the effects of the approximation in the microgrid state equations.
Still, the combined effect of these two terms is minimal, in practice.

In Figure~\ref{fig:convergence} we compared the behavior of the two considered clustering strategies.
One can notice different things. 
First, both strategies converge to the same minimum, which is slightly larger than the minimum losses that could be achieved by solving the original, nonconvex optimization problem. As expected, the clustering choice do not affect the steady state performance of the algorithm.
Second, the performance of the edge-disjoint gossip algorithm results indeed to be better than the star topology in the long-time regime, as the analysis in Section \ref{sec:optimalHypergraph} suggests, and its slope corresponds to the fastest achievable
rate of convergence,
as predicted by Theorem \ref{thm:bestbeta}.
Third, one can see that, while the performance metric that we adopted 
is meaningful for describing the long-time regime,
it does not describe, in general, the initial stage (or short-time behavior).
Indeed, in the initial transient, the full dynamics of the system \eqref{eq:updateStep} contribute to the algorithm behavior.
This is evident in the star-topology strategy.
In the edge-disjoint case, on the other hand, all the eigenvalues of 
the systems \eqref{eq:updateStep} result to be the same, and thus the adopted performance metric (the slowest eigenvalue) fully describes the algorithm evolution.
This fact suggests that a different metric can possibly be adopted, 
in order to better characterize the two regimes.
Choosing the most appropriate metric requires a preliminary derivation of a dynamic model for the variation of reactive power demands of the loads in time.
In the case of time-varying loads, different strategies will corresponds to different steady-state behaviors of the algorithm.
We expect that an $L_2-$type metric will be needed for the characterization of the performance of different clustering choices, while the slowest eigenvalue will not be informative enough in this sense.
A preliminary study in this sense has been proposed in \cite{Bolognani2012a}.

\begin{figure}[tb]
\raggedleft
\includegraphics[scale=0.94]{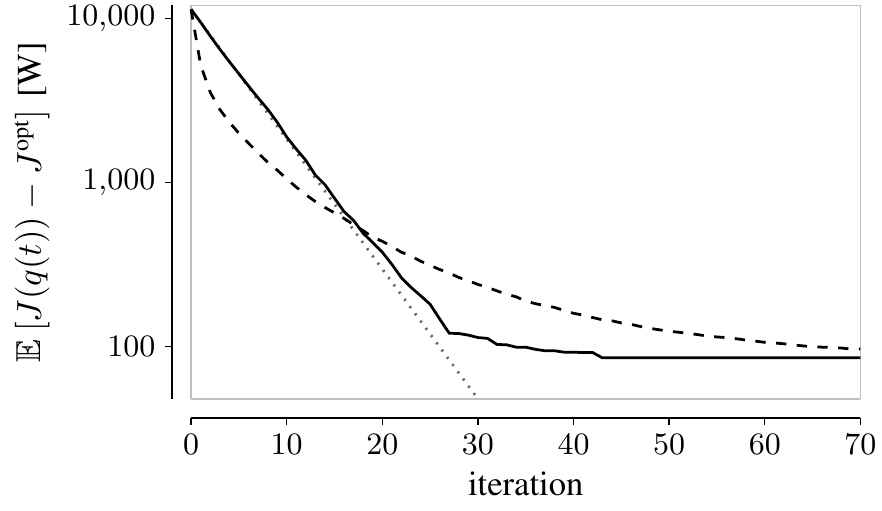}$\ $
\caption{Simulation of the expected behavior of the algorithm for different clustering strategies.
The difference between power losses at each iteration of the algorithm 
(averaged over $1000$ realizations)
and the minimum possible losses $J^\text{opt}$,
has been plotted for 
two different clustering choices: \emph{edge-disjoint gossip} (solid line) and \emph{star topology} (dashed). 
The slope represented via a dotted line represent the best possible rate of convergence that the algorithm can achieve,
according to Theorem \ref{thm:bestbeta}.}
\label{fig:convergence}
\end{figure}

\section{Conclusions}

The proposed model for the problem of optimal 
reactive power compensation in smart microgrids exhibits two main features. 
First, it can be casted into the framework of quadratic optimization, 
for which robust solvers are available and the performance analysis becomes tractable; 
second, it shows how the physics of the system can be exploited to design a distributed algorithm 
for the problem.

On the basis of the proposed model, we have been able to design a distributed, leaderless and randomized algorithm
for the solution of the optimization problem, requiring only local communication and local knowledge of the network topology.

We then proposed a metric for the performance of the algorithm, for which we are able to provide a bound on the best 
achievable performances. We are also able to tell which clustering choice is capable of giving the 
optimal performances. It is interesting that the optimal strategy requires short-range communications: 
this is a notable consequence of the presence of an underlying physical system, considered that consensus algorithm (which share many features with the proposed algorithm) benefit from long-range communication that shorten the graph diameter. 

At the same time, this result is also motivating from the technological point of view, 
because achieving cooperation of inverters that are far apart in the microgrid
presents a number of issues
(e.g. time synchronization, limited range of power line communication technologies).

Motivated by the final remarks of Section \ref{sec:simulations}, 
we plan to add two elements to the picture in the next future:
the network dynamics (both the dynamic behavior of the power demands, 
and the electrical response of the system when actuated), 
and the presence of operational constraints for the compensators, 
whose capability of injecting reactive power is limited and possibly time varying.




\appendices

\section{Proof of Proposition \ref{proposition:taylor_expansion_complex}}
\label{app:taylor}

We first introduce the new variable $\epsilon:=1/U_N$. 
We also assume, with no loss of generality, that the phase of the voltage at the PCC is $\phi = 0$.
This way we can see the currents $i$ and the voltages $u$ as functions $i(\epsilon),u(\epsilon)$ of $\epsilon$, determined by the equations
\begin{equation}
\begin{cases}
u=\green  i+ \epsilon^{-1}\1\\
\1^T i = 0 \\
u_v\bar i_v = s_v\left|\epsilon u_v\right|^{\eta_v},\quad \forall v \in \nodes \backslash \{ \PCC \}\\
\end{cases}
\label{eq:det_iu_eps}
\end{equation}
We construct the Taylor approximation of $u(\epsilon)$ and $i(\epsilon)$ around $\epsilon=0$.

Let us define 
\begin{equation}
\begin{split}
\delta_v(\epsilon) &:=  
 i_v(\epsilon)\epsilon^{-1}-\bar s_v\\
\lambda_v(\epsilon) &:=  
 u_v(\epsilon)\epsilon^{-1} - \epsilon^{-2} - [\green \bar s]_v 
\end{split}
\label{eq:deltalambda}
\end{equation}
and substitute them into (\ref{eq:det_iu_eps}).
We get a system of equations in $\delta,\lambda,\epsilon$ which can be written as
$$
\begin{cases}
G_v(\delta,\lambda,\epsilon)=0,\quad \forall v \in \nodes\\
F_v(\delta,\lambda,\epsilon)=0,\quad \forall v \in \nodes 
\end{cases}
$$
where
\begin{flalign*}
G_0(\delta,\lambda,\epsilon)&:=\lambda_0&\\
F_0(\delta,\lambda,\epsilon)&:=\sum_{w\in\nodes}\delta_w&
\end{flalign*}
and, for all $v \ne \PCC$, 
\begin{flalign*}
G_v(\delta,\lambda,\epsilon)&:=\lambda_v-\sum_{w\in\nodes}[\green]_{vw} \delta_w,&\\ 
F_v(\delta,\lambda,\epsilon)&:= 
\left(
1+\epsilon^2\lambda_v+\epsilon^2\sum_{w\in\nodes}[\green]_{vw}\bar s_w
\right)
(s_v+\bar\delta_v) \\
&\phantom{:=}- s_v
\left|
1+\epsilon^2\lambda_v+\epsilon^2\sum_{w\in\nodes}[\green]_{vw}\bar s_w
\right|^{\eta_v}.&
\end{flalign*}

We have $2n$ complex equations in $2n+1$ complex variables. 

We interpret now the complex numbers as vectors in $\realnumbers^2$ and 
$G_v,F_v$  as functions from $\realnumbers^{2n}\times \realnumbers^{2n}\times \realnumbers^2$ to $\realnumbers^{2}$.

Observe that 
$G_v(\delta,\lambda,\epsilon)_{|\delta=0, \lambda=0,\epsilon=0}=0$
and 
$F_v(\delta,\lambda,\epsilon)_{|\delta=0, \lambda=0,\epsilon=0}=0$ 
for all $v\in\nodes$.
By the implicit function theorem, if the matrix
$$\begin{bmatrix}
\left(\displaystyle\frac{\partial G_v}{\partial \delta_w}\right)_{v,w\in\nodes}
&
\left(\displaystyle\frac{\partial G_v}{\partial \lambda_w}\right)_{v,w\in\nodes}
\\ \\
\left(\displaystyle\frac{\partial F_v}{\partial \delta_w}\right)_{v,w\in\nodes}
&
\left(\displaystyle\frac{\partial F_v}{\partial \lambda_w}\right)_{v,w\in\nodes}
\end{bmatrix}
$$ 
evaluated in $\delta=0, \lambda=0, \epsilon=0$ is invertible,
then $\delta(\epsilon),\lambda(\epsilon)$ have continuous derivatives in $\epsilon=0$.

In order to determine this matrix we need to introduce the following notation. If $a\in\mathbb C$, and $a=a^R+ja^I$ with $a^R,a^I\in\mathbb R$, then we define
$$\langle a\rangle:=\left[\begin{array}{cc} a^R&-a^I\\a^I&a^R\end{array}\right]$$
With this notation observe that, if $a,x\in\mathbb C$ and we consider those complex numbers as  vectors in $\mathbb R^2$ and if we consider the function $f(x):=ax$ as a function from $\mathbb R^2$ to $\mathbb R^{2}$, then we have that $\partial f/\partial x=\langle a\rangle$. Notice moreover that the the function $g(x):=\bar x$ has 
$$\frac{\partial g}{\partial x}=N:=\left[\begin{array}{cc} 1&0\\0&-1\end{array}\right]$$
From these observations we can argue that
\begin{align*}
\frac{\partial G_0}{\partial \delta_w}=0,
& \qquad\qquad
\frac{\partial G_0}{\partial \lambda_w}=\begin{cases}
I\quad\text{if $w=0$}\\
0\quad\text{if $w\not =0$}
\end{cases}
\\
\frac{\partial F_0}{\partial \delta_w}=I,
& \qquad\qquad
\frac{\partial F_0}{\partial \lambda_w}=0
\end{align*}
Instead for all $v\not=0$ we have that
$$\frac{\partial G_v}{\partial \lambda_w}=\begin{cases}
I\quad\text{if $w=v$}\\
0\quad\text{if $w\not =v$}
\end{cases}
$$
while
$$
\frac{\partial G_v}{\partial \delta_w}=\langle -\green_{vw}\rangle\qquad \forall w\in\nodes.
$$
Observe finally that, if $w\not =v$, then $\partial F_v/\partial \lambda_w=\partial F_v/\partial \delta_w=0$ and that
$$
\frac{\partial F_v}{\partial \delta_v}=
\left\langle
1+\epsilon^2\lambda_v+\epsilon^2\sum_{w\in\nodes}[\green]_{vw}\bar s_w 
\right\rangle N
$$
which evaluated in $\epsilon=0$ yields the matrix $N$.
On the other hand,
$$
\frac{\partial F_v}{\partial \lambda_v}=
\langle \epsilon^2 (s_v+\bar\delta_v) \rangle
- s_v\frac{\partial \left|1+\epsilon^2\lambda_v+\epsilon^2\sum_{w\in\nodes}[\green]_{vw}\bar s_w\right|^{\eta_v}}{\partial \lambda_v}
$$
where $s_v$ is a $2$-dimensional column vector and 
$\frac{\partial \left|1+\lambda_v\right|^{\eta_v}}{\partial \lambda_v}$ is a $2$-dimensional row vector.
With some easy computations it can be seen that
$
\frac{\partial \left|1+\lambda_v\right|^{\eta_v}}{\partial \lambda_v}
$,
evaluated in $\epsilon=0$, yields $[0\ 0]$.
Hence in $\epsilon=0$ we have 
$$
\frac{\partial F_v}{\partial \lambda_v} = \langle 0\rangle.
$$
We can conclude that for $\delta=0, \lambda=0, \epsilon=0$ we have
\begin{multline*}
\left[\begin{array}{cccccc}
\left(\displaystyle\frac{\partial G_v}{\partial \delta_w}\right)_{v,w\in\nodes}&\left(\displaystyle\frac{\partial G_v}{\partial \lambda_w}\right)_{v,w\in\nodes}\\
 \\
\left(\displaystyle\frac{\partial F_v}{\partial \delta_w}\right)_{v,w\in\nodes}&\left(\displaystyle\frac{\partial F_v}{\partial \lambda_w}\right)_{v,w\in\nodes}
 \\
\end{array}
\right]=
\\
\left[\begin{array}{cccc|cccc}
0&0&\cdots&0&I&0&\cdots&0\\ 
0&\langle [\green]_{11} \rangle &\cdots&\langle [\green]_{1,n-1} \rangle
&0&I&\cdots&0\\
\vdots&\vdots&\ddots&\vdots&\vdots&\vdots&\ddots&\vdots\\
0&\langle [\green]_{n-1,1} \rangle&\cdots&\langle [\green]_{n-1,n-1} \rangle&0&0&\cdots&I\\
\hline
I&I&\cdots&I&0&0&\cdots&0\\
0&N&\cdots&0&0&0&\cdots&0\\
\vdots&\vdots&\ddots&\vdots&\vdots&\vdots&\ddots&\vdots\\
0&0&\cdots&N&0&0&\cdots&0\\
\end{array}
\right]
\end{multline*}
which is invertible.
By applying Taylor's theorem,
we can thus conclude from \eqref{eq:deltalambda} that
\begin{align*}
i_v &= \frac{\bar s_v}{U_N} + \frac{c_v(U_N)}{U_N^2}\\
u_v &= U_N + \frac{[\green \bar s]_v}{U_N} + \frac{d_v(U_N)}{U_N^2},
\end{align*}
where $c_v(U_N)$ and $d_v(U_N)$ are bounded functions of $U_N$.

\section{Proof of Theorem~\ref{thm:betaLessThanR}}
\label{app:convergence}

In order to prove~Theorem \ref{thm:betaLessThanR}, we need the following technical results.

\begin{lemma}
Consider the matrices $F_\cli$, $\cli=1,\ldots,\ell$, as defined in \eqref{eq:Fdefinition}. The following statements are equivalent:
\begin{enumerate}[i)]
\item the point $x=0$ is the only point in $\ker \1^T$ such that
$F_\cli x = x$ for all $\cli = 1,\ldots,\ell$;
\item $\Imag[\Omega_1 \ldots \Omega_\ell] =  \ker \1^T$;
\item the hypergraph $\mathcal H$ is connected.
\end{enumerate}
\label{lem:threeways}
\end{lemma}
\begin{proof}
Let us first prove that ii) implies i).
Assume that $\Imag[\Omega_1 \ldots \Omega_\ell] = \ker \1^T$ and that 
$x\in\ker \1^T$ is such that  $F_\cli x=x$ for all $\cli$.
Then we can find $y_\cli\in\mathbb R^n$ such that  
$$
x = \sum_\cli \Omega_\cli y_\cli.
$$
Moreover, as $F_\cli x = x$ for all $\cli$, then $(\Omega_\cli M \Omega_\cli)^\sharp M x = 0$ 
and so, by the properties mentioned above, we can argue that $M x \in \ker \Omega_\cli$.
We therefore have
$$
x^T M x = \sum_\cli y_\cli^T\Omega_\cli M x=0,
$$
and so, since $M$ is positive definite, it yields that $x=0$.

We prove now that i) implies ii).
The inclusion $\Imag[\Omega_1 \ldots \Omega_\ell] \subseteq \ker \1^T$ is always true and follows from the fact that $\Imag\Omega_\cli \subseteq \ker \1^T$.
We need to prove only the other inclusion.
Suppose that $x=0$ is the only point in $\ker \1^T$ such that  $F_\cli x=x$ for all $i$. 
This means that
$$
\ker \1^T\cap \ker (I-F_1)\cap\cdots\cap\ker (I-F_\ell)=\{0\},
$$
which implies that
$$\realnumbers^\nocompensators=
\Imag \1 + \Imag (I-F_1^T)+\cdots+\Imag (I-F_\ell^T)$$
Take now any $x \in \ker \1^T$. Then there exist a $\alpha\in\realnumbers$ and $y_i\in\realnumbers^\nocompensators$ such that
\begin{align*}
M x	&=\alpha\1+\sum_{\cli=1}^\ell(I-F_\cli)^Ty_\cli
		=\alpha\1+\sum_{\cli=1}^\ell M(\Omega_\cli M \Omega_\cli)^\sharp y_\cli
\end{align*}
Then $\alpha\1=Mw$ where
$$
w = x-\sum_{\cli=1}^\ell(\Omega_\cli M \Omega_\cli)^\sharp y_\cli=x-\sum_{\cli=1}^\ell \Omega_\cli z_\cli
$$
where we have used the fact that $\Imag (\Omega_\cli X \Omega_\cli)^\sharp=\Imag\Omega_\cli$. From the fact that $\Imag\Omega_\cli\subseteq\ker \1^T$, we can argue that $w\in\ker \1^T$ and so it follows that $0=w^T\alpha\1=w^TMw$. Finally, since $M$ is positive definite on the subspace $\ker \1^T$, we can conclude that $w=0$ and so that  
$x \in \Imag[\Omega_1 \ldots \Omega_\ell]$. 

We finally prove that ii) and iii) are equivalent.
We consider the matrix $W\in \mathbb R^{m\times m}$ with entries $[W]_{hk}$ equal to the number of the sets $\mathcal C_\cli$ which contain both $h$ and $k$. Let $\mathcal G_W$ be the weighted graph associated with $W$.
It is easy to see that the hypergraph having  $\mathcal C_\cli$ as edges
is connected if and only if $\mathcal G_{W}$ is a connected graph.

Let us define $\chi_{\mathcal C_\cli}:\compensators \rightarrow \{0,1\}$ 
as the characteristic function of the set $\mathcal C_\cli$,
namely a function of the nodes that is $1$ when the node belongs to $\mathcal C_\cli$ and is zero otherwise. 
Then the Laplacian  $L_{W}$ of $\mathcal G_{W}$ can be expressed as follows
\begin{align*}
L_{W} &= \sum_{h,k \in \compensators}(\1_h-\1_k)(\1_h-\1_k)^T [W]_{hk}\\
&= \sum_{h,k \in \compensators}(\1_h-\1_k)(\1_h-\1_k)^T \sum_{i=1}^\ell \chi_{\mathcal C_\cli}(h)\chi_{\mathcal C_\cli}(k) \\
&= \sum_{\cli=1}^\ell\sum_{h,k \in \mathcal C_\cli}(\1_h-\1_k)(\1_h-\1_k)^T 
= \sum_{\cli=1}^\ell 2|\mathcal C_\cli| \Omega_\cli \\
&= [\Omega_1 \ldots \Omega_\ell] \diag\{2|\mathcal C_1|I,\ldots,2|\mathcal C_\ell| I\}
\left[\begin{array}{c}\Omega_1\\ \vdots\\ \Omega_\ell\end{array}\right] .
\end{align*}
Notice that the graph connectivity is equivalent to the fact that $\ker L_W=\Imag \1$ and so, by the previous equality, it is equivalent to 
$$
\ker \left[\begin{array}{c}\Omega_1\\ \vdots\\ \Omega_\ell\end{array}\right] =\Imag \1
$$
which is equivalent to ii).
\end{proof}

We also need the following technical lemmas.

\begin{lemma}
Let $P, Q \in \mathbb R^{m\times m}$ and $P \ge Q$. Then $\mathcal L^k(P) \ge \mathcal L^k(Q)$ for all $k \in \mathbb Z_{\ge 0}$.
\label{lem:monotoneL}
\end{lemma}

\begin{proof}
From the definition of $\mathcal L$, we have
\begin{equation*}
\begin{split}
x^T\left[ \mathcal L(P) - \mathcal L(Q) \right] x 
&= x^T \left[ \expect{F_\cli^T P F_\cli} - \expect{F_\cli^T Q F_\cli} \right] x \\
&=  \expect{x^T F_\cli^T (P - Q) F_\cli x}  
\ge 0,
\end{split}
\end{equation*}
and therefore $P \ge Q$ implies $\mathcal L(P) \ge \mathcal L(Q)$. 
By iterating these steps $k$ times we then obtain $\mathcal L^k(P) \ge \mathcal L^k(Q)$.
\end{proof}

\begin{lemma}
\label{lem:removingOmegas}
For all $\Delta$ we have that
$
\Omega \mathcal L^k (\Omega \Delta \Omega) \Omega = \Omega \mathcal L^k (\Delta) \Omega.
$
\end{lemma}
\begin{proof}
Proof is by induction. The statement is true for $k=0$, as $\Omega^2 = \Omega$. Suppose it is true up to $k$.
We then have
\begin{align*}
\Omega \mathcal L^{k+1} (\Delta) \Omega 
&= \Omega \mathcal L ( \mathcal L^{k}(\Delta)) \Omega 
= \Omega \mathcal L ( \Omega \mathcal L^{k}(\Delta) \Omega ) \Omega \\
&= \Omega \mathcal L ( \Omega \mathcal L^{k}( \Omega \Delta \Omega) \Omega ) \Omega 
= \Omega \mathcal L^{k+1}( \Omega \Delta \Omega) \Omega.
\end{align*}
\end{proof}

\begin{lemma}
All the eigenvalues of $\Fave$ are real and have absolute value not larger than $1$.
If $\Imag\begin{bmatrix}\Omega_1 \cdots \Omega_\ell\end{bmatrix} = \ker \1^T$
and if Assumption \ref{ass:randomSubproblemExecution} holds,
then the only eigenvalue of $\Fave$ on the unitary circle is $\lambda = 1$,
with multiplicity $1$ and with associated left eigenvector $\1$ and right eigenvector $M^{-1}\1$.
\label{lem:eigenvaluesOfFBar}
\end{lemma}

\begin{proof}
Recall that the matrices $F_\cli$ are projection operators, i.e. $F_\cli^2 = F_\cli$ and so they have eigenvalues $0$ or $1$. Recall moreover that  $F_\cli$ are self-adjoint matrices with respect to the inner product $\langle \cdot, \cdot \rangle_M$,  defined as $\langle x, y \rangle_M = x^T M y$, 
This implies that  $||F_\cli||_M\le 1$ where $||\cdot||_M$ is the induced matrix norm with respect to the vector norm $||x||_M:=\langle x, x \rangle_M^{1/2}$. This implies that
$$
\|\Fave\|_M=\|\expect{F_\cli}\|_M \le \expect{\|F_\cli\|_M} \le 1
$$
And so, since also $\Fave$ is self-adjoint, its eigenvalues are real and are smaller than or equal to $1$ in absolute value.

Assume now that $\Fave x = \lambda x$, with $|\lambda|=1$. Then we have
$$
\|x\|_M = \|\Fave\|_M = \|\expect{F_\cli} x\|_M \le \expect{\|F_\cli x\|_M} \le \|x\|_M.
$$
If Assumption \ref{ass:randomSubproblemExecution} holds 
(and thus the probabilities $\rho_\cli$'s are all strictly greater than $0$),
then the last inequality implies that $\|F_\cli x\|_M = \|x\|_M$ for all $\cli$'s.
As $F_\cli$ are projection matrices, it means that $F_\cli x = x$ and so $Mx \in \ker \Omega_\cli^T, \forall \cli$. 
Using the fact that $\Imag\begin{bmatrix}\Omega_1 \cdots \Omega_\ell\end{bmatrix} = \ker \1^T$, we necessarily have $x = M^{-1}\1$. By inspection we can verify that the left eigenvector corresponding to the same eigenvalue is $\1$.
\end{proof}

We can now give the proof of Theorem \ref{thm:betaLessThanR}.

\begin{proof}[Proof of Theorem \ref{thm:betaLessThanR}]
Let us first prove that $\Omega \mathcal L (\Omega M \Omega) \Omega \le \beta \Omega M \Omega$.
Indeed, we have
\begin{equation*}
\begin{split}
x^T \Omega \mathcal L(\Omega M \Omega) \Omega x &=
\expect{x^T \Omega F_\cli^T \Omega M \Omega F_\cli \Omega x} \\
&= \expect{x^T \Omega F_\cli^T M F_\cli \Omega x} \\
&= x^T \Omega M^{1/2} \expect{M^{1/2} F_\cli M^{-1/2}} M^{1/2} \Omega x,
\end{split}
\end{equation*}
where we use the fact that $\Omega F_\cli \Omega = F_\cli \Omega$ and $F_\cli^T M F_\cli= M F_\cli$.
Notice moreover that 
$\expect{M^{1/2} F_\cli M^{-1/2}} = M^{1/2} \Fave M^{-1/2}$ is symmetric and, 
by Lemma~\ref{lem:threeways} ane Lemma~\ref{lem:eigenvaluesOfFBar}, it has only one eigenvalue on the unit circle (precisely in $1$), with eigenvector $M^{-1/2}\1$.
As $M^{1/2} \Omega x \perp M^{-1/2}\1$ for all $x$, we have
\begin{equation*}
x^T \Omega \mathcal L(\Omega M \Omega) \Omega x \le \beta  x^T \Omega M \Omega x,
\end{equation*}
with $\beta = \max \{ |\lambda| \;|\; \lambda \in \lambda(\Fave), \lambda \ne 1\}$.
From this result, using Lemmas \ref{lem:monotoneL} and \ref{lem:removingOmegas}, we can conclude
\begin{equation*}
\begin{split}
\Omega \mathcal L^t (\Omega M \Omega) \Omega 
&= \Omega \mathcal L^{t-1} \left( \mathcal L (\Omega M \Omega ) \right) \Omega \\
&= \Omega \mathcal L^{t-1} \left( \Omega \mathcal L (\Omega M \Omega ) \Omega \right) \Omega \\
& \le \Omega \mathcal L^{t-1} \left( \beta \Omega M \Omega \right) \Omega \\
&= \beta \Omega \mathcal L^{t-1} \left( \Omega M \Omega \right) \Omega 
\le \cdots \le \beta^t \Omega M \Omega,
\end{split}
\end{equation*}
and therefore $R \le \beta$.
\end{proof}

\bibliographystyle{IEEEtran}

\end{document}